\documentclass[11pt]{article}

\RequirePackage[OT1]{fontenc}
\RequirePackage{amsthm,amsmath}

\usepackage[square,numbers]{natbib}
\RequirePackage[colorlinks,citecolor=blue,urlcolor=blue]{hyperref}
\usepackage[dvipsnames]{xcolor}
\usepackage{verbatim}
\usepackage{xcolor}
\usepackage{fullpage}
\usepackage{subfig}
\usepackage{layout}
\usepackage{dsfont}
\usepackage[mathscr]{eucal}
\usepackage[toc,page]{appendix}
\usepackage{mathrsfs}
\usepackage{color}
\usepackage{pifont}
\usepackage{bm}
\usepackage{latexsym}
\usepackage{amsfonts}
\usepackage{amssymb}
\usepackage{epsfig}
\usepackage{graphicx}
\usepackage{multirow}
\usepackage{hyperref}
\usepackage{enumitem}
\usepackage{amsmath}
\usepackage{amsthm}

\usepackage{units}
\usepackage{xfrac}  
\usepackage{breqn}
\usepackage{dsfont}
\usepackage{graphicx}
\usepackage{breqn}

\usepackage{enumitem, hyperref}
\makeatletter
\def\namedlabel#1#2{\begingroup
  #2%
  \def\@currentlabel{#2}%
  \phantomsection\label{#1}\endgroup
}
\makeatother

\newcommand {\Rbb}{\mathbb{R}}
\newcommand {\Pbb}{\mathbb{P}}
\newcommand {\Ebb}{\mathbb{E}}
\newcommand {\Hbb}{\mathbb{H}}

\newcommand {\Id}{I}

\newcommand {\Ncal}{\mathcal{N}}

\newcommand {\Fcal}{\mathcal{F}}
\newcommand {\Bcal}{\mathcal{B}}
\newcommand {\Tcal}{\mathcal{T}}

\newcommand {\Bscr}{\mathscr{B}}

\newcommand {\Wscr}{\mathscr{W}}

\newcommand {\CovSpace}{\mathscr{K}}
\newcommand {\CovSpacePos}{\mathscr{K}_+}

\newcommand{\diff}{\,\textnormal{d}}

\newcommand {\Acov}{A}
\newcommand {\Fcov}{F}
\newcommand {\Gcov}{G}

\newcommand{\G}[1]{\mathcal{G}_{#1}}
\newcommand{\CFG}{C}
\renewcommand{\H}{H_n}
\newcommand{\X}{X}

\newcommand {\tr}{\mathrm{tr}}
\newcommand {\trace}{\tr}
\newcommand{\range}{\mathrm{range}}
\newcommand{\Frechet}{Fréchet }
\newcommand{\Tan}{\Tcal}
\newcommand{\der}{d}

\newcommand{\bary}{\Xi}
\newcommand{\empbary}{{\widehat\Xi_n}}
\newcommand{\empbaryk}{{\widehat\Xi_{n_k}}}
\newcommand{\empbarykl}{{\widehat\Xi_{n_{k_l}}}}

\theoremstyle{plain}
\newtheorem{theorem}{Theorem}
\newtheorem{lemma}{Lemma}
\newtheorem{proposition}{Proposition}
\newtheorem{corollary}{Corollary}
\theoremstyle{remark}
\newtheorem{remark}{Remark}

\theoremstyle{plain}

\usepackage{hyperref}
\hypersetup{ colorlinks = true, linkcolor  = blue}

\numberwithin{equation}{section}
  \usepackage[affil-it]{authblk}

\makeatletter
\newcommand{\subjclass}[2][1991]{%
  \let\@oldtitle\@title%
  \gdef\@title{\@oldtitle\footnotetext{#1 \emph{Mathematics subject classification.} #2}}%
}
\newcommand{\keywords}[1]{%
  \let\@@oldtitle\@title%
  \gdef\@title{\@@oldtitle\footnotetext{\emph{Key words and phrases.} #1.}}%
}
\makeatother
\usepackage{abstract}

   \title{{Large Sample Theory for Bures--Wasserstein Barycentres 
 }}

\author{
  Leonardo V.\ Santoro
  \qquad 
  Victor M. Panaretos   \\ {\footnotesize{\texttt{leonardo.santoro@epfl.ch}  \qquad\,\,\texttt{victor.panaretos@epfl.ch}}}}
  
  \affil{Institut de Math\'ematiques\\École Polytechnique Fédérale de Lausanne}

\begin{document}

\maketitle

\begin{abstract}
We establish a strong law of large numbers and a central limit theorem in the Bures-Wasserstein space of covariance operators -- or equivalently centred Gaussian measures -- over a general separable Hilbert space. Specifically, we show that empirical barycentre sequences indexed by sample size are almost certainly relatively compact, with accumulation points comprising population barycentres. We give a sufficient regularity condition for the limit to be unique. When the limit is unique, we also establish a central limit theorem under a refined pair of moment and regularity conditions. 
 Finally, we prove strong operator convergence of the empirical optimal transport maps to their population counterparts.
Though our results naturally extend finite--dimensional counterparts, including associated regularity conditions, our techniques are distinctly different owing to the functional nature of the problem in the general setting. A key element is the characterisation of compact sets in the Bures-Wasserstein topology that reflects an \emph{ordered} Heine-Borel property of the Bures--Wasserstein space.

\medskip
    \noindent \textbf{MSC2020 classes:}  60B12, 60G57, 60H25, 62R20, 62R30   

\end{abstract}

\begin{footnotesize}
\tableofcontents
\end{footnotesize}

\section{Introduction}

Covariance operators encode the second-order variation structure of a stochastic process, and are ubiquitous in the theory, application and computational analysis thereof. In particular, their spectral analysis provides a canonical means to decompose the associated random process into uncorrelated and orthogonal components, via the celebrated Karhunen-Lo\`eve expansion. Consequently, they are fundamental tools for statistical inference on random functions, including estimation, testing, regression, and classification.  Increasingly, covariance operators are becoming the main target of statistical interest in several contexts. For instance when several different processes, corresponding to different conditions or populations, are to be contrasted. Or, when the random phenomenon under study gives rise to outcomes/realisations of covariance form, such as in the analysis of diffusion tensors in medical imaging, or in  the study of functional connectivity in neuroscience. Trace-one covariance operators also naturally arise in mathematical physics and quantum information, in the form of density operators that describe the statistical state of a quantum system. Section~\ref{sec : applications} provides a more detailed discussion of several motivating applications for the study of collections of infinite-dimensional covariance operators.

\medskip

At an abstract level, collections of covariance operators can be seen as corresponding to realisations of a random element $\Sigma$ in the space of trace-class nonnegative definite operators $\CovSpace=\CovSpace(\Hbb)$ acting on a Hilbert space $\Hbb$. Given a metric $\Pi$ on $\CovSpace$, a barycentre (or \Frechet mean) $\bary$ of $\Sigma$ is a minimiser of the functional 
$$
 F \mapsto \Ebb [\, \Pi(\Sigma, F)^{2}\,].
$$
Similarly, given a sample $\Sigma_1, \ldots, \Sigma_n$, consisting of independent and identically distributed (i.i.d.) copies of $\Sigma$, an empirical barycentre (or empirical \Frechet mean) $\widehat{\bary}_n$ is a minimiser of the empirical functional
$$
 F\mapsto \frac{1}{n} \sum_{i=1}^{n} \Pi\left(\Sigma_{i}, F\right)^{2}.
$$
Questions of existence and uniqueness of such \Frechet means $\bary$ are heavily dependent on the nature of $\Pi$ and $\Hbb$. The same is true of the large sample ($n\to\infty$) theory relating the empirical \Frechet mean $\widehat{\bary}_n$ to its population counterpart $\bary$.

These questions are well understood when the Hilbert space $\Hbb$ is finite-dimensional, i.e. when dealing with covariance matrices on $\mathbb{R}^d$. They are also well understood for infinite-dimensional $\Hbb$, provided the metric $\Pi$ is induced by a norm, for example by embedding $\CovSpace$ in the (Hilbert) space of Hilbert-Schmidt operators or the (Banach) space of trace class operators. Nevertheless, it is often preferable from a statistical point of view, to contrast covariances in an ``intrinsic" metric, rather than via linear embedding.

When the Hilbert space $\Hbb$ is infinite-dimensional and we use an ``intrinsic" metric $\Pi$, the combined non-linear \emph{and} functional structure of $\CovSpace$ elicits subtle difficulties that escape existing results. The purpose of this paper is to establish key results for barycentres in this case, including unique existence and large sample theory (Law of Large Numbers and Central Limit Theorem). We focus specifically on the \emph{Bures-Wasserstein metric} on $\CovSpace$,
 $$\Pi(\Gcov_1,\Gcov_2)^2 =  \tr(\Gcov_1)+\tr(\Gcov_2)-2\sqrt{\tr(\Gcov^{\sfrac{1}{2}}_2\Gcov_1\Gcov_2^{\sfrac{1}{2}})},\qquad \Gcov_1,\Gcov_2\in\CovSpace.$$
{When $\mathrm{dim}(\Hbb)<\infty$ this metric arises in quantum information theory, where it is used as a metric on density operators, i.e. the nuclear unit ball of $\CovSpace$ (see e.g. \citet{bhatia2019bures}). It induces an Alexandrov geometry with singularities, where the singular set is stratified according to the dimension of the range of the corresponding covariance matrices  (\citet{takatsu2010wasserstein}}). Contrary to several intrinsic covariance metrics which are defined only for $\mathrm{dim}(\Hbb)<\infty$, the Bures-Wasserstein metric extends readily to the general case $\mathrm{dim}(\Hbb)=\infty$, because it does not rely on quantities such as determinants or inverses. 

The use of this metric in an infinite-dimensional statistical context  was initiated (under a different expression) by \citet{pigoli2014distances}, who extended the finite-dimensional Procrustes metric, which is popular in statistical shape theory and has previously been seen to hold advantages in the analysis of covariance \emph{matrices}  \citep{dryden2009non}. It was then shown by \citet{masarotto2019procrustes} and \citet{bhatia2019bures}, in finite and infinite dimensions respectively, that the Procrustes metric between two covariances coincides with the 2-Wasserstein metric between the corresponding centred Gaussian processes -- the infinite-dimensional analogue of the Bures--Wasserstein metric. The implied connection to optimal transport further highlighted the canonical role of this metric for statistics on covariance operators, and elucidated certain key properties, at least at the sample level. However, questions relating to the ``population" barycentre, and especially questions pertaining to large sample theory, remained open.

In this paper, we establish results concerning existence, uniqueness of the \Frechet means (Theorem~\ref{thm : existence and uniqueness}). We establish a strong law of large numbers (Theorem~\ref{thm : consistency of Frechet mean}), by constructing a class of compact subsets of the Bures--Wasserstein space, which may be of independent interest. We link Bures-Wasserstein barycentres to a certain fixed-point equation (Proposition~\ref{prop : fixed point}) and address their \textit{regularity} (Proposition \ref{prop : bary positive}). We then establish the pointwise (strong operator topology) consistency of empirical optimal transport maps to their population counterparts (Theorem~\ref{thm : convergence of optimal maps}). Finally, establish a Central Limit Theorem  for \Frechet means (Theorem~\ref{thm : clt}).  While such results are arguably of interest in their own right, they are also basic to statistical inference on covariance operators, where the nature of problems is inherently infinite-dimensional (see Section~\ref{sec : applications}).

Our existence/uniqueness results naturally generalise finite-dimensional counterparts, including associated regularity conditions. However, the problem of limit theory is fundamentally more challenging in the infinite dimensional case, and our analysis necessarily relies on a distinctly different approach. This is because arguments/conditions that are essential to the finite dimensional limit theory are inherently incompatible with general Hilbert spaces. For instance, \citet{le2017existence} treat existence and consistency (law of large numbers) in the context of (measures on) \emph{locally compact spaces}, which distinctly rules out the case of (measures on) general Hilbert spaces. Similarly, \citet{evans2020strong} treat the consistency of \Frechet mean sets (in the absence of uniqueness) on abstract metric spaces that satisfy the Heine-Borel property, i.e.\ for which closed and bounded sets are compact. As for rates of convergence, \citet{le2022fast} establish the convergence rate of the empirical to the population barycentre, but under a minimal eigenvalue gap condition, fundamentally incompatible with the infinite dimensional case. \citet{kroshnin2021statistical}, on the other hand, study central limit theory, but use regularity results for Gaussian transport maps in $\Hbb=\mathbb{R}^d$, $d<\infty$, which are intrinsically finite-dimensional and do not hold true in infinite-dimensional Hilbert spaces.

\section{Motivating Applications}\label{sec : applications}

In this section, we briefly survey some areas where collections of covariance operators on infinite dimensional Hilbert spaces naturally arise as a main object of interest.

\paragraph*{Multiple Functional Populations}
A  natural statistical application in which covariance operators are the main object of interest typically occurs in situations where several different ``populations" of functional data are considered, with each population characterized by distinct structural features.
Such variation may in fact be present at the level of the mean (leading to what has been known as Functional Analysis of Variance), as well as at the level of the covariance operators.
That is, each
$k = 1, \dots,K$ of the $K$ populations arises as $n_k$ copies of a random function $X_k$ with population specific mean $\mu_k$ and covariance operator $\Sigma_k$. In other words, we observe $K$ independent samples $\{X_{i,1}\}_{i=1,\dots,n_1}, \dots, \{X_{i,k}\}_{i=1,\dots,n_k}$ of functions with well defined respective means $\{\mu_i\}_{i=1,\dots,K}$ and covariances $\{\Sigma_i\}_{i=1,\dots,K}$.
The classical functional ANOVA makes the assumption that the second order variations be the same across the populations, and investigates differences at the level of the mean. However, one might also be interested in discerning whether the
different groups of functions manifest the same type of dispersion \textit{relative} to
their mean. This amounts to understanding how strongly collections of covariances deviate, relative to some notion of mean theoref. This problem has received a lot of attention in the statistical literature. For instance, \citet{benko2009common} probed for similarities/differences in the stochastic behaviour of (log returns of) implied volatilities, for different maturities.  \citet{panaretos2010second} investigate how the mechanical properties (seen as stochastic fluctuations) of short strands of DNA might relate to their based-par sequences.  \citet{ferraty2006nonparametric} compare surfaces representing log spectrograms of short spoken words spoken by different individuals.  \citet{pigoli2014distances} study variations in speech frequency/intensity within different languages. For an in depth discussion and more examples we refer to \cite{masarotto2022transportation} and discussion therein.

\paragraph*{FMRI and Functional Connectivity. } An important domain in neuroscience is the study of functional connectivity. Here, one essentially observes longitudinal samples of families of covariances,  obtained via functional Magnetic Resonance Imaging (fMRI). This data measure brain activity by detecting changes associated with blood flow  for each subject in a sample (\citet{lindquist2008statistical}), and understanding the corresponding covariance structure provides insight into the brain’s functional organization and its modelling. In fact, we have recently observed an explosive growth in the number of neuroimaging studies performed using fMRIs: see for instance \citet{dai2019analyzing} -- who analyse dynamical brain functional connectivity as trajectories on space of covariances, or \citet{petersen2019frechet} -- who assess the co-development of brain regions in infants as a correlation matrix that varies with age. In principle, the phenomena under study evolve in the continuum, and in the limit concern covariance operators. It is important to therefore have statistical methodology that either directly applies to the case of operators, or at least guarantees that existing methods will scale stably with finer resolution.

\paragraph*{Functional Time Series and Covariance Flows.} Infinite dimensional covariance operators naturally arise in the analysis of stationary functional time series (FTS, e.g. \citet{koner2023second}) in the form of the (power) spectral density operator: a collection of covariance operators indexed by frequency (see \citet{panaretos2013fourier}). These constitute the Fourier transform of the time-lagged autocovariance operators of the FTS, and allow for the spectral representation of the FTS via its Cram\'er-Karhunen-Lo\`eve representation, which serves as the starting point for many statistical tasks. Each functional time series is thus associated to a flow of covariance operators, which can be seen as a process in the Bures-Wasserstein space (e.g. \citet{santoro2023random}). The study of such processes and the development of corresponding statistical methods and theory rests importantly on the avaialbility of large sample theory in the Bures-Wasserstein space.

\paragraph*{Quantum Information} 
In quantum information theory, covariance operators with unit trace play a central role as they represent density operators, which describe the statistical state of a quantum system. Quantum probability adopts the formalism of self-adjoint operators as a non-commutative analogue of classical random variables and probability measures (e.g., \citet{gill2004teleportation}). Physical systems can be either finite or infinite-dimensional, and our limit theory facilitates the extension of results to general quantum observables.
For a given set of density matrices, quantum state averaging is traditionally achieved through linear averaging methods. However, as \citet{kroshnin2021statistical} observe, the Bures-Wasserstein barycenter offers an alternative approach, defining the "most typical" representative state based on the fidelity measure provided by the Bures-Wasserstein metric.

\section{Background and Notation}

\subsection{Preliminaries}
Let $\Hbb$ be a separable Hilbert space with inner product $\langle\cdot,\cdot\rangle \::\: \Hbb\times\Hbb\rightarrow\Rbb$ and induced norm $\|\cdot \|\::\: \Hbb \rightarrow \Rbb_{+}$, with $\mathrm{dim}(\Hbb)\in \mathbb{N}\cup\{\infty\}.$
We denote by $\Bscr$ the space of bounded linear operators $\Acov\::\: \Hbb\rightarrow \Hbb$. For $f,g\in\Hbb$, the \textit{tensor product} $f\otimes g \::\: \Hbb\rightarrow\Hbb$ is the linear operator defined by:
$$
(f\otimes g)u = \langle g,u\rangle f,\qquad u\in\Hbb.
$$ 
For a linear operator $\Acov:\Hbb\rightarrow\Hbb$, we define its adjoint as the unique operator $\Acov^*$ such that $\langle \Acov u, v \rangle = \langle u, \Acov^* v \rangle$ for all $u,v\in \Hbb$. We say that an operator $\Acov$ is \textit{self-adjoint} if $\Acov=\Acov^*$. We say that $\Acov$ is \textit{non-negative}, and write $\Acov\succeq 0$ if it is self-adjoint, and satisfies $\langle\Acov h, h\rangle \geq 0$ for all $u\in\Hbb$. {We say that $\Acov$ is \textit{positive} or \emph{regular}, and write $\Acov\succ 0$ if it is self-adjoint, and satisfies $\langle\Acov h, h\rangle >0 $ for all $u\in\Hbb\setminus\{0\}$ (the reason for the term ``regular" will become clear in the next subsection). } We say that $\Acov$ is compact if for any bounded sequence $\{h_n\}\subset  \Hbb$, $\{\Acov h_n\}\subset  \Hbb$ contains a convergent sub-sequence.
If $\Acov$ is a non-negative, \textit{compact} operator, then there exists a unique non-negative operator denoted by $\Acov^{\sfrac{1}{2}}$ that satisfies $( \Acov^{\sfrac{1}{2}} )^2 = \Acov$.
The \textit{kernel} of $\Acov$ is denoted by $\ker(\Acov) = \{h\in\Hbb\::\: \Acov h= 0\}$, and its \textit{range} by $\range(\Acov)=\{\Acov h \::\: h \in \Hbb\}$. 
We denote the \textit{trace} of an operator $\Acov$, when defined, by $\tr(\Acov)= \sum_{i\geq 1}\langle \Acov e_i,e_i\rangle $, where $\{e_i\}_{i\geq 1}$ is a Complete Orthonormal System (CONS) of $\Hbb$. We write
 $$
 \|\Acov\|_\infty:=\sup_{\|h\|=1}\|\Acov h\|,\quad 
 \|\Acov\|_2:= \sqrt{\tr(\Acov^*\Acov)}, \quad 
 \|\Acov\|_1:= \tr(\sqrt{\Acov^*\Acov})
 $$
 for the operator norm,   Hilbert-Schmidt norm and trace norm, respectively. For any operator $\Acov$, one has
 $$
  \|\Acov\|_\infty \leq
 \|\Acov\|_2\leq
 \|\Acov\|_1.
 $$
 We denote by $\Bscr,\Bscr_1$ and $\Bscr_2$ the space of
 bounded, trace-class and Hilbert-Schmidt  linear operators on $\Hbb$, respectively.
 Note that if $\Acov \in \Bscr_1$ and $\Acov\succeq 0$, then $\Acov^{\sfrac{1}{2}}\in \Bscr_2$ and  $
 \|\Acov^{\sfrac{1}{2}}\|_2^2 = \|\Acov\|_1
 $.
We write $\Id$ for the identity operator on $\Hbb$. 

\medskip

 We denote by $\CovSpace$  the space of non-negative definite and trace-class operators on $\Hbb$,
 $$\CovSpace\equiv \CovSpace(\Hbb)=\{\Fcov: \Hbb\to \Hbb : \Fcov\succeq 0\,\, \&\,\, \|F\|_1<\infty\},$$
and by $\CovSpacePos$ its subset of strictly positive elements.

 \medskip

Let us recall the concept of \textit{Bochner} integral of measurable functions taking values in normed spaces, following \citep[][Definitions 2.6.1-3]{hsing2015theoretical} 
Let $(E,\Omega,\mu)$ be a measurable space and $(\Bcal,\|\cdot\|_{\Bcal})$ a Banach space.
A function $f: E \rightarrow \Rbb$ is called simple if it can be represented as
$
f(\cdot)=\sum_{i=1}^k 1_{E_i}(\cot) g_i
$
where $\{E_i\}_{i=1,\dots,k}\subset E$ are measurable and $\{g_i\}_{i=1,\dots,k}\subset \Bcal$. A simple function is Bochner integral if $\sum_{i=1}^k\mu(E_i)<\infty$, in which case its \textit{Bochner} integral of a simple function $f$ is  defined as
$$
\int_E f d \mu :=\sum_{i=1}^k \mu\left(E_i\right) g_i
$$
This definition does not depend on the particular representation of $f$. More generally, any measurable function $f$ is said to be Bochner integrable if there exists a sequence $\left\{f_n\right\}$ of simple and Bochner integrable functions such that
$
\lim _{n \rightarrow \infty} \int_E\left\|f_n-f\right\| d \mu=0,
$
in which case the Bochner integral of $f$ is defined as 
$$
\int_E f d \mu:=\lim _{n \rightarrow \infty} \int_E f_n d \mu
$$

\medskip

Let us also recall the notion of \Frechet differentiability.
Let $\Bcal_1,\Bcal_2$ be Banach spaces, and consider the operator $\Phi\::\:\Bcal_1 \rightarrow \Bcal_2$
We say that $\Phi$ is \Frechet differentiable at $F \in \Bcal_1$ if there exists an open set $A\subset \Bcal_1$ containing $F$ and a linear operator $\mathcal{D}_F\::\:A \rightarrow \Bcal_2$ such that:
$$
\frac{\| \Phi(F+X)- \Phi(F) - \mathcal{D}_F (X) \|_{\Bcal_2}}{\|X\|_{\Bcal_1}} = o(1), \quad \text{as } \|X\|_{\Bcal_1}\to 0.
$$
In such case, we call $\mathcal{D}_F$ the \Frechet derivative or differential of $\Phi$ at $F$. Higher order derivatives are defined recursively.





\subsection{Optimal Transportation of Gaussian Measures}

\noindent The space $\CovSpace(\Hbb)$ of non-negative definite trace-class operators on $\Hbb$ can be identified with the space of all centred Gaussian measures on $\Hbb$. Indeed, if $X$ is zero-mean Gaussian process valued in $\Hbb$ satisfying $\mathbb{E}\|X\|^2<\infty$, we have $\Ebb [ X\otimes X ] = \Sigma$ and $\mathbb{E}\|X\|^2=\Ebb [ \tr ( X \otimes X) ]=\|\Sigma\|_1<\infty$. In light of this identification, it is natural to equip $\CovSpace$ with a metric, $\Pi$, borrowed from the theory of optimal mass transport (see \citet{villani2009optimal} for an extensive introduction). That is, for covariances $\Fcov, \Gcov$, we can consider Gaussian measures $\mu_{\Fcov} \equiv \Ncal(0,\Fcov)$ and $\nu_{\Gcov} \equiv \Ncal(0,\Gcov)$, and define a metric $\Pi$ on $\CovSpace$ via
\begin{equation}
    \label{eq : Bures--wasserstein distance}
    \Pi(\Fcov,\Gcov) :=  W_2(\mu_{\Fcov},\nu_{\Gcov}).
\end{equation}
Here, $W_2(\mu,\nu)$ denotes the 2-Wasserstein distance between the measures $\mu, \nu$:
\begin{equation}\label{eq : wasserstein distance}
    W_{2}(\mu, \nu)=\left( \inf _{\pi \in \Gamma(\mu, \nu)} \int_{\Hbb \times \Hbb}\|x-y\|^{2} \diff \pi(x, y)\right)^{\sfrac{1}{2}},
\end{equation}
where the infimum is taken over the set $\Gamma(\mu,\nu)$ of all couplings of $\mu$ and $\nu$. The optimisation problem appearing in \eqref{eq : wasserstein distance} is known as the Monge-Kantorovich problem of optimal transportation (\citet{kantorovich1942translocation}).
A simple compactness argument shows that the infimum in the Monge-Kantorovich problem is always attained by some coupling $\pi$, for any marginal pair of measures $\mu, \nu \in \Wscr_2(\Hbb)$. Moreover, when $\mu$ is  \emph{regular},
the optimal coupling is unique and given by a deterministic coupling, in the sense that there exists some deterministic map $T_{\mu}^{\nu}: \Hbb \rightarrow \Hbb$, called an optimal (transportation) map, for which $\pi=\left(\Id, T_{\mu}^{\nu}\right) \# \mu$ is a minimiser. {For a precise definition of a \emph{regular} measure $\mu$ on a general $\Hbb$, see \citet[Def. 6.2.2]{ambrosio2005gradient}. For instance, if $\Hbb = \mathbb{R}^d$, then a sufficient condition for existence of an optimal map is that $\mu$ does not give mass to $d-1$ dimensional sets. Regardless of $\Hbb$ being finite or infinite-dimensional, a Gaussian measure $\mu$ is \emph{regular} if and only if its covariance operator is strictly positive-definite. This explains why strictly positive-definite elements of $\CovSpace$ are equivalently called \emph{regular}.}

\bigskip

Gaussians have an exceptional place within Wasserstein spaces, as distances between Gaussians admit a closed form expression. Indeed, for $\mu \equiv \Ncal(0,\Fcov)$ and  $\nu \equiv \Ncal(0,\Gcov)$, we may express:
\begin{equation}\label{def : distance}
\Pi(\Fcov,\Gcov) =  W_2(\mu_{\Fcov},\nu_{\Gcov}) = \sqrt{ \tr(\Fcov)+\tr(\Gcov)-2\tr(\Gcov^{\sfrac{1}{2}}\Fcov\Gcov^{\sfrac{1}{2}})^{\sfrac{1}{2}}}
\end{equation}
See \citet{olkin1982distance} for a proof in finite dimensional spaces, and \citet{cuesta1996lower} in any separable Hilbert space. 
Furthermore, optimal maps also admit a closed-form representation: writing $T_{\Fcov}^{\Gcov}$ in lieu of the optimal transport map $T_{\mu}^{\nu}$, it may then be shown that:
\begin{equation}
    \label{eq : transport map}
    T_{\Fcov}^{\Gcov} := 
        \Fcov^{-\frac{1}{2}} (\Fcov^{\frac{1}{2}}\Gcov \Fcov^{\frac{1}{2}})^{\sfrac{1}{2}}\Fcov^{-\frac{1}{2}}.
\end{equation}
whenever such map exists. 
When $\Hbb=\mathbb{R}^{d}$ is finite-dimensional, invertibility of $\Fcov$ will guarantee the existence and uniqueness of a deterministic optimal coupling of $\mu \equiv \Ncal\left(0, \Fcov\right)$ of $\nu \equiv$ $\Ncal\left(0, \Gcov\right)$, induced by the linear transport map \eqref{eq : transport map}. The same formula turns out to be essentially valid in infinite dimensional Hilbert spaces, provided that the ranges of $\Fcov$ and $\Gcov$ are appropriately nested: when $\operatorname{ker}\left(\Fcov\right) \subseteq \operatorname{ker}\left(\Gcov\right)$, the map \eqref{eq : transport map} is well defined on a subspace of $\Hbb$ with $\mu$--measure 1 (see e.g. \citet[Proposition 2.2]{cuesta1996lower}). It is fairly straightforward to see that this subspace contains the range of $\Fcov^{\sfrac{1}{2}}$. However, the linear map $T_{\Fcov}^{\Gcov}$ is generally an unbounded operator and \textit{cannot} be extended to the whole of $\Hbb$.

\subsection{The Bures--Wasserstein Space}
The space of covariance operators $\CovSpace$ equipped with the metric $\Pi(\cdot,\cdot)$ is known in the literature as \textit{Bures--Wasserstein} space. 
The Bures--Wasserstein distance has appeared in the literature 
 in many different settings. In fact, it was observed by \citet{bhatia2019bures, masarotto2019procrustes}  that the Bures--Wasserstein metric coincides with the Procrustes metric for covariance operators, as considered in \citet{pigoli2014distances}:
\begin{equation}
\label{eq : Pi - procrustes equivalence}
     \Pi(\Fcov,\Gcov) = \inf_{U\::\:U^*U=\Id} \| \Fcov^{\sfrac{1}{2}} - \Gcov^{\sfrac{1}{2}}U\|_2,\qquad \Fcov,\Gcov\in\CovSpace.
\end{equation}
Moreover, whenever the transport map $T_\Fcov^\Gcov$ exists, it is easily seen that such distance allows for an equivalent expression in terms of transport maps:
\begin{equation}
\label{eq : Pi - in terms of maps}
   \Pi(\Fcov, \Gcov)^2 \: = \: \|(T_\Fcov^\Gcov-\Id)\Fcov^{\sfrac{1}{2}} \|_2^2 = \|(T_\Fcov^\Gcov-\Id)\Fcov (T_\Fcov^\Gcov-\Id) \|_1,
\end{equation}

In fact, the topology induced by $\Pi$ is equivalent to the topology generated by the trace norm distance, as well as the Hilbert-Schmidt distance on square roots, as summarised in the following lemma.

\begin{lemma}[Bures--Wasserstein topology]\label{lemma : topology}
Let $\Fcov,\Gcov \in \CovSpace$. 

\begin{enumerate}

    \item[1)] $ \| \Fcov^{\sfrac{1}{2}} - \Gcov^{\sfrac{1}{2}} \|_2^2\leq \| \Fcov-\Gcov\|_1 \leq (\tr(\Fcov)^{\sfrac{1}{2}} + \tr(\Gcov)^{\sfrac{1}{2}}) \Pi(\Fcov , \Gcov ) $

    \item[2)] $ \Pi(\Fcov , \Gcov ) \leq \| \Fcov^{\sfrac{1}{2}} - \Gcov^{\sfrac{1}{2}} \|_2 \leq  \| \Fcov - \Gcov \|_1^{\sfrac{1}{2}}  .$

\end{enumerate}
\end{lemma}
\begin{proof}[Proof of Lemma~\ref{lemma : topology}]
The first inequality in {1)}, relating the trace norm and the Hilbert-Schmidt norms, is known as Powers–St{\o}rmer inequality: see \citet{Stormer1970}. The second inequality was proved in \citet[][Proposition 4]{masarotto2019procrustes}.
The second equation {2)} is straightforward consequence of the Powers–St{\o}rmer's inequality (\citet{Stormer1970}) and of \eqref{eq : Pi - procrustes equivalence}.
\end{proof} 

The above lemma readily shows that the space of covariance operators equipped with the metric $\Pi$ is \textit{complete}. In particular, the limit of a weakly convergent sequence of Gaussian measures is Gaussian, and the class of Gaussian measures is closed in the $2$-Wasserstein space. 

\bigskip

The Bures-Wasserstein space of covariance operators $(\CovSpace,\Pi)$ inherits the distinctive geometric structure of the Wasserstein space $\Wscr_2(\Hbb)$. When restricting the latter to the closed set of Gaussian measures in $\Wscr_2(\Hbb)$, one can define a formal \textit{tangent space} at any $\Sigma\in\CovSpace$ as
$$\Tan(\Sigma)=\overline{\{\Gamma: \Gamma=\Gamma^*\,\, \& \,\, \|\Sigma^{1/2}\Gamma\|_2<\infty\}}.$$
The closure is taken with respect to the corresponding norm $\|\Gamma\|_{\Tan(\Sigma)}:=\|\Sigma^{1/2}\Gamma\|_2$, which is a Hilbert norm. We refer to \citet{takatsu2010wasserstein} and \citet{masarotto2019procrustes}, in finite and infinite dimensions respectively.

\section{Random Covariance Operators and their \Frechet Means}
\label{sec : random cov}

 The notion of a \emph{mean} strongly depends on the metric of the ambient space. For normed spaces, the mean of a random element $\chi$ is defined by ``linear averaging", i.e. the Bochner integral: $\int \chi \diff P$, where $P$ denotes the law of $\chi$
(\citet[][Definition 7.2.1]{hsing2015theoretical}.
For random elements in general (non-linear) metric spaces, such as $(\CovSpace,\Pi)$, one cannot directly integrate, unless the non-linear space is embedded in larger Banach space. 
In the case of $\CovSpace$, this approach has unavoidable weaknesses, as it embeds covariances in a larger linear space in which they are not closed
under linear operations, potentially distorting the intrinsic geometric and topological properties of the original space.
Instead, one defines a \textit{\Frechet mean} -- first introduced by \citet{frechet1948elements} -- which proceeds indirectly, by the property of a mean as a centre of mass, a.k.a.\ \emph{a barycentre}. In our setting, let  $\Sigma$ be random element of $\CovSpace$, i.e. a measurable mapping from some probability space $(\Omega,\mathfrak{F},\Pbb)$ to the measurable space $\CovSpace$ equipped with the Borel $\sigma$-field generated by $\Pi$. A (population) \Frechet mean or (population) barycentre of $\bary$ is defined as any minimiser of the expected squared distance,
\begin{equation}
\label{eq : pop frechet mean}
    \Ebb [\, \Pi(\Sigma, \bary)^{2}\,] \leq \Ebb [\, \Pi(\Sigma, \Fcov)^{2}\,],\qquad \forall\, \Fcov\in\CovSpace.
\end{equation}
Given an i.i.d. collection $\Sigma_1,...,\Sigma_n$ of copies of $\Sigma$, an empirical \Frechet mean (or empirical barycentre) $\widehat{\bary}_n$ is also naturally defined as any minimiser of the empirical average
\begin{equation}\label{eq : sample frechet mean}
   \frac{1}{n}\sum_{i=1}^{n}\Pi(\Sigma_i, \widehat{\bary}_n)^{2} \leq \frac{1}{n}\sum_{i=1}^{n}\Pi(\Sigma_i,\Fcov)^{2},\qquad \forall\, \Fcov\in\CovSpace.
\end{equation}

Though \Frechet means have historically been primarily studied on Riemannian manifolds (see \citet{bhattacharya2003large}, \citet{afsari2011riemannian} among others), the generality of their definition allows for very broad use, and \Frechet means are increasingly becoming of interest in statistical problems involving Wasserstein spaces (see \citet{zemel2019frechet}, \citet{petersen2019frechet}, \citet{chen2021wasserstein} among others).

\medskip

The study of \Frechet means is typically plagued with geometrical complications. \Frechet means often fail to be unique, contrary to the ``linear" mean, and even existence is not guaranteed beyond the realm of locally compact spaces (\citet{kendall2011limit}).
In the $d$-dimensional 2-Wasserstein space $\Wscr(\Rbb^d)$, sufficient conditions for existence of empirical \Frechet means  were obtained the seminal work by \citet{agueh2011barycenters}. 
Specifically for the Bures--Wasserstein space (over $\Hbb=\mathbb{R}^d$), \citet{agueh2011barycenters} linked the empirical \Frechet mean to a certain matricial fixed point equation. \citet{masarotto2019procrustes} established existence and conditions for uniqueness for empirical \Frechet means in the Bures--Wasserstein space over a general separable Hilbert space $\Hbb$.   \citet{kroshnin2021statistical}, on the other hand established existence, uniqueness and regularity for population \Frechet means in the finite dimensional Bures--Wasserstein space ($\mathbb{R}^d$), under minimal conditions. Unfortunately, their proofs and arguments are intrinsically finite-dimensional, and do not generalise to the general case of an infinite dimensional Hilbert space $\Hbb$.

\medskip 

 To derive our results we will make use of some conditions on the law underlying the random covariance $\Sigma\in\CovSpace$, which involve both \emph{integrability} and \emph{regularity} aspects. 

\begin{enumerate}

    \item[\namedlabel{A1}{\textbf{A(1)}}]  $\Ebb\Big[ \|\Sigma\|_1\Big] < \infty. $

        \item[\namedlabel{A2}{\textbf{A(2)}}] $\Pbb\{\Sigma \succ 0 \} > 0$.
    \end{enumerate}

\noindent Assumption~\ref{A1} is the natural version of an $L^2$-type moment condition on $\CovSpace$, in light of the topology induced by $\Pi$. Assumption~\ref{A2} is a regularity condition, assigning positive probability to the event that $\Sigma$ be regular. When a \Frechet mean  $\bary$ of $\Sigma$ exists, 
we will make use of a second set of conditions involving a more refined control of the first moment and the regularity:

    \begin{enumerate}
    
    \item[\namedlabel{B1}{\textbf{B(1)}}]  
        $\Ebb \Big[ \left\| T_{\bary}^{\Sigma}\right\|_{\Tan_{\bary}} \Big] < \infty$.
    \end{enumerate}

    \begin{enumerate}
    
        \item[\namedlabel{B2}{\textbf{B(2)}}]     There exist  $c,C:\CovSpace\to (0,\infty)$ and a fixed covariance $R\in \CovSpacePos$ such that we have $$\Pbb\left( c(\Sigma)R \preceq \Sigma \preceq C(\Sigma)R\right) = 1$$ with $\Ebb[C(\Sigma)]<\infty.$

    \end{enumerate} 

\noindent We now comment on these assumptions in more detail.

Assumptions~\ref{A1} and~\ref{A2}  grant the existence and uniqueness of the \Frechet mean $\bary$, respectively, as will be shown in Theorem~\ref{thm : existence and uniqueness}. They are natural, in the sense of being the population analogues of the assumptions that 
grant existence and uniqueness of a sample \Frechet mean on a general $\Hbb$ (see \citet[Cor. 9 \& Prop. 10]{masarotto2019procrustes}) and for a population \Frechet mean when $\Hbb=\mathbb{R}^d$ (see \citet[][Theorem 2.1]{kroshnin2021statistical}).


When a \Frechet mean $\bary$ is well-defined, Assumption~\ref{B2} will be made use of to ensure that the unique barycenter is regular, and to obtain a Central Limit Theorem, in the form of Theorem~\ref{thm : clt}.

Assumption~\ref{B1} is a central first moment assumption on the tangent space of the \Frechet mean. It states that the mean absolute deviation of $\Sigma$ (interpreted on the tangent space, as the image of $\Sigma$ via the log map at $\bary$) is finite in the natural norm. 
 Note that~\ref{B1} is verified in the classical ``template deformation model" from geometrical statistics, which in our setting models $\Sigma$ as a random perturbation of a latent template, $$\Sigma = T\bary T$$ for $T\succeq 0$ a random optimal map satisfying $\mathbb{E}\|T\|_{\infty}<\infty$ and $\mathbb{E}[T]=I$.  
 
Assumption~\ref{B2} is a more refined regularity condition than~\ref{A2}. To see this, note that~\ref{B2} is equivalent to stating that the range of $\Sigma^{1/2}$ is invariant over the support of its law -- i.e. every realisation of $\Sigma$ generates the same Reproducing Kernel Hilbert Space (RKHS), almost surely. This further illustrates why~\ref{B2} can be seen as a regularity assumption. {Indeed, we will see that it is connected with a notion of compactness (see Lemma~\ref{lemma : charac of compact}) which implies regularity in the analytical sense in infinite dimensions}. When $\Hbb$ is finite dimensional,~\ref{B2} can be related to the notion of \emph{uniquely extendable geodesics} which is used by \citet{le2022fast} to obtain a rate of convergence for empirical barycentres.

Finally, we comment on some relationships between the conditions we have stated. 
 It is easy to see that~\ref{B2} implies assumption~\ref{A1}, and~\ref{A2}. In fact, when $\mathrm{dim}(\Hbb)=\infty$, the following holds:
\begin{lemma}
    If $\dim(\Hbb) < \infty$, then assumption~\ref{B2} holds if and only if
    \begin{equation}\label{eqn:E-R-equiv}
        \Ebb\Big[ \|\Sigma\|_1\Big] < \infty\qquad\textnormal{ and }\qquad \Pbb(\Sigma\succ 0) = 1.
    \end{equation}
\end{lemma}

\begin{proof}
    First, suppose that assumption~\ref{B2} holds, and let us show \eqref{eqn:E-R-equiv} by noting
    \begin{equation*}
        \Ebb\left[\Pi^2(\Sigma,0)\right] = \Ebb\left[\trace(\Sigma)\right] \le \Ebb\left[\trace(C(\Sigma)R)\right] = \Ebb\left[C(\Sigma)\right]\cdot\trace(R) < \infty
    \end{equation*}
    and also $P\{\Sigma: \Sigma\succ 0\} \ge P\{\Sigma: c(\Sigma) > 0\} = 1$.
    Second, let us suppose that \eqref{eqn:E-R-equiv} holds, and let us show~\ref{B2}.
    Since $\dim(\Hbb)<\infty$, we can set $R:=I$ and set $c(\Sigma)$ and $C(\Sigma)$ to be the smallest and largest eigenvalues of $\Sigma$, respectively.
    We of course have $\Pbb\left(c(\Sigma)I \preceq \Sigma \preceq C(\Sigma)I\right) = 1$ by construction, and we can compute:
    \begin{equation*}
        \Ebb\left[C(\Sigma)\right] \le \Ebb\left[\trace(\Sigma)\right],
    \end{equation*}
    and the right side is finite by assumption.
\end{proof}

Let us also comment on the link between~\ref{A1} and~\ref{B1}. Indeed,~\ref{B1} is \textit{weaker} in terms of moments, as it implies that $\int_{\CovSpace}\|\Sigma\|_1^{\sfrac{1}{2}}\diff P(\Sigma) <\infty$; however, it is \textit{stronger} in terms of regularity, since under~\ref{A1} the optimal map appearing in~\ref{B1} is not guaranteed to exist.

Finally, we note that~\ref{B1} is implied by~\ref{B2}. The proof is however not immediate, and we postpone it to Lemma~\ref{lemma:boundbothmaps}.

\subsection{Existence and Uniqueness}
\label{subsec existence uniqueness and characterisation of frechet means}
The following theorem provides sufficient conditions for existence and uniqueness of the \Frechet mean of a random covariance operator on a general separable Hilbert space, possibly infinite--dimensional. Interestingly, no additional assumptions are required compared to the analogous result in the finite-dimensional case (see the assumptions in \citet[][Theorem 2.1]{kroshnin2021statistical}), and so it is a genuine generalisation. The proof of the result is however distinctly different, and relies on classical functional analysis tools.

\begin{theorem}\label{thm : existence and uniqueness}
Let $\Sigma$ be random element of $\CovSpace(\Hbb)$ for $\Hbb$ a separable (possibly infinite dimensional) Hilbert space. A \Frechet mean of $\Sigma$ with respect to the Bures--Wasserstein distance $\Pi$ will exist if and only if $\Sigma$ satisfies~\ref{A1}. If $\Sigma$ additionally satisfies~\ref{A2},
the \Frechet mean is unique. 
\end{theorem}

\begin{proof}[Proof of Theorem~\ref{thm : existence and uniqueness}]

Denote by $\Psi(F)=\mathbb{E} [\Pi(\Sigma, F)^{2}]$ the \Frechet functional.
The squared Wasserstein distance is weakly convex along generalised geodesics, as shown in the proof of \citep[][Proposition 10]{masarotto2019procrustes}. Indeed:
\begin{equation}
\label{eq:convexityPidist}    
\Pi^2(\lambda F_1 + (1-\lambda)F_2,\Sigma) \leq \lambda \Pi^2(F_1,\Sigma) + (1-\lambda)\Pi^2(F_2,\Sigma), \qquad \lambda\in[0,1], F_1,F_2\in\CovSpace
\end{equation}
and consequently $\Psi(\cdot)$ is convex along generalised geodesics as well, provided all terms involved are a.s. finite. This is guaranteed by~\ref{A1}.

Secondly, note that $\Psi(\cdot)$ is coercive.
 Indeed, first observe that $\Psi(0) = \Ebb\tr(\Sigma)<\infty$. Then, we may see that if $\{F_n\}_{n \geq 1}$ is a sequence of operators that diverges in trace, i.e.\ $\tr(F_n)\rightarrow\infty$, necessarily $\Psi(F_n) \rightarrow \infty$. Indeed, by \eqref{eq : Pi - procrustes equivalence}, for every $n$ there exists a unitary operator $U_n$ such that $\Pi (F_n,\Sigma) = \|F_n^{\sfrac{1}{2}} - U_n\Sigma^{\sfrac{1}{2}}\|_2$, and consequently:
 $$
\Psi(F) \geq \Ebb \lvert \|F_n^{\sfrac{1}{2}}\|_2 - \|U_n\Sigma^{\sfrac{1}{2}} \|_2 \rvert ^2
\geq \Ebb  (  \|F_n^{\sfrac{1}{2}}\|_2 - \|\Sigma^{\sfrac{1}{2}} \|_2 )^2 
=
\Ebb  (  \|F_n\|^{\sfrac{1}{2}}_1 - \|\Sigma \|^{\sfrac{1}{2}}_1 )^2 
 $$
which clearly diverges. Note that we have used that $\|U_n \Sigma\|_1\leq \|U_n\|_\infty\|\Sigma\|_1$ and  $\|U_n\|_\infty = 1$.
This shows that any minimising sequence $\{F_n\}_{n\geq 1}$ is bounded in trace norm. In particular, for any sequence of non-negative definite operators
 $\|F_n\|_1 \rightarrow \infty \Leftrightarrow\|F_n^{\sfrac{1}{2}}\|_2 \rightarrow \infty $,
 by which we obtain that any minimising sequence $\{F_n^{\sfrac{1}{2}}\}_{n\geq 1}$ is bounded in the space $\Bscr_2$ of Hilbert-Schmidt operators. Moreover, since closed balls in $\Bscr_2$ are weakly compact, there exists a sub-sequence $\{F^{\sfrac{1}{2}}_{n_k}\}_{k\geq 1}$ weakly-converging the to some limit $F^{\sfrac{1}{2}}$, which need be non-negative, self-adjoint and Hilbert-Schmidt. Finally, since $\Psi(\cdot)$ is convex and continuous, it is (sequentially) weakly l.s.c.\ (see \citet[][Corollary 3.9 ]{brezis2011functional}), and thus
$$
\inf_{\omega\in\CovSpace} \Psi(\omega) 
= \lim_{k\rightarrow\infty} \Psi((F^{\sfrac{1}{2}}_{n_k})^2) \geq \Psi(F)
$$
 proving that $F$ is a minimum, i.e.\ a \Frechet mean, which exists.

\medskip

Next, we show that whenever~\ref{A2} is satisfied, the \Frechet functional is \textit{strictly convex}. 
\begin{align*}
    \Ebb &\left[\Pi^2(\lambda  F_1 + (1-\lambda)F_2,\Sigma)\right]= \\
     &\quad=\: \Ebb\left[ \Pi^2(\lambda F_1 + (1-\lambda)F_2,\Sigma) \:|\: {\Sigma \succ 0} \right]\Pbb({\Sigma \succ 0} ) 
    \\ &\quad\qquad\qquad+ \Ebb\left[ \Pi^2(\lambda F_1 + (1-\lambda)F_2,\Sigma) \:|\: {\Sigma \nsucc 0} \right]\Pbb({\Sigma \nsucc 0} ) \\
      &\quad< \: \Ebb\left[ \lambda \Pi^2(F_1, \Sigma) + (1-\lambda)\Pi(F_2,\Sigma) \:|\: {\Sigma \succ 0} \right]\Pbb({\Sigma \succ 0} ) \\ &\qquad\qquad + \Ebb\left[ \lambda \Pi^2(F_1, \Sigma) + (1-\lambda)\Pi(F_2,\Sigma) \:|\: {\Sigma \nsucc 0} \right]\Pbb({\Sigma \nsucc 0} )\\
     &\quad=  \: \lambda \Ebb\left[\Pi^2(F_1,\Sigma)\right] + (1-\lambda)\Ebb\left[\Pi^2(F_2,\Sigma)\right],
\end{align*}
where we have used that \eqref{eq:convexityPidist}, that the Bures-Wasserstein distance from a \textit{positive}, fixed covariance is strictly convex, and that $\Pbb(\Sigma \succ 0)>0$. This shows \textit{strict} convexity of $\Psi(\cdot)$, and hence uniqueness of the \Frechet minimiser. 
\end{proof}

Note that~\ref{A1} is necessary for the well definition of the \Frechet functional.
Assumption~\ref{A2} is \emph{not} necessary for uniqueness, as can be verified by the following simple example in finite-dimensions: take $\Hbb=\mathbb{R}^2$ and let 
$$\Sigma=\left(\begin{array}{cc}W & 0 \\0 & 0\end{array}\right),\quad W\sim \chi^2_1.$$ 
For this reason, we will henceforth directly assume uniqueness/existence wherever required, without explicitly invoking~\ref{A2}, unless we specifically need to. However, when not satisfied, the barycentre can indeed fail to be unique, as shown in \citet{PennecGeod}.

\subsection{Law of Large Numbers}\label{subsec : estimation bary}

Recall that an empirical or sample \Frechet mean of $\Sigma_{1}, \cdots, \Sigma_{n}$ is a minimiser of the empirical \Frechet functional, see \eqref{eq : sample frechet mean}. Such a mean always exists and is furthermore unique, provided one of the $\Sigma_i$ is regular (see \citet[][Corollary 9]{masarotto2019procrustes}). An important line of enquiry is the (strong) consistency of   empirical \Frechet means for their population counterparts, as the sample size diverges. Equivalently, establishing a strong law of large numbers in the metric space $(\CovSpace,\Pi)$.  In view of the results by \citet{ziezold1977expected}, if a population mean exists and the sequence of empirical means converge, then the limit must be the population mean, under uniqueness. Therefore, the main step for obtaining consistency  amounts to showing that, with probability 1, a sequence of empirical barycentres $\{\empbary\}_{n\geq 1}$ is precompact, with population \Frechet means as limit points.

\medskip

{We do this in Theorem~\ref{thm : consistency of Frechet mean} using the following characterisation of compact subsets of the Bures--Wasserstein space, which may be of interest in its own right. }

\begin{lemma}\label{lemma : charac of compact}
Let $ \Fcal\subset \CovSpace$ be  tight. The set:
\begin{equation}\label{eq : charac compact set}
  \mathfrak{D}_{\Fcal} := \{Q\in \CovSpace \;:\; 0 \preceq Q\preceq F \text{ for some } F\in\Fcal\}
\end{equation}
is compact in the topology of $\Pi$ (equivalently, in the trace-norm topology).
\end{lemma}
Note that the lemma in particular implies that for any fixed covariance $F\in\CovSpace$,  the set
$
\{Q\in \CovSpace \;:\; 0 \preceq Q\preceq F\},
$
which can be thought of as ``closed and bounded interval" in the partial order relation ``$\preceq$", is compact in the $\Pi$-topology, establishing a sort of ordered Heine-Borel property for the Bures-Wasserstein space.
Also note that a converse statement is clearly true, and hence that Lemma~\ref{eq : charac compact set} provides in fact a characterisation of compact sets in $(\CovSpace,\Pi)$.

\begin{proof}[Proof of Lemma~\ref{lemma : charac of compact}]
Employing the duality between covariance operators and centred Gaussian measures, its clear that the set $\mathfrak{D}_{\Fcal}$ in \eqref{eq : charac compact set} is compact if and only if the corresponding family of Gaussian measures
\begin{equation}\label{eq : family of gaussian measures}
    \{\mu_{Q}\equiv \Ncal(0,Q)\;:\; Q\in \mathfrak{D}_{\Fcal}\}
\end{equation} 
is tight. 
Indeed,  tightness of \eqref{eq : family of gaussian measures} implies -- via Prokhorov's theorem -- the (sequential) relative compactness of the family of measures \eqref{eq : family of gaussian measures}  with respect to the topology of weak convergence; since all measures in \eqref{eq : family of gaussian measures} are Gaussian, this is in turn equivalent to compactness with respect to the topology induced by the Wasserstein distance, i.e. to the compactness of \eqref{eq : charac compact set} with respect to the Bures--Wasserstein distance $\Pi(\cdot,\cdot)$. 

Tightness of \eqref{eq : family of gaussian measures} can be established by proving its \textit{flat concentration}, which in turn can be shown employing the criterion in \citet[][Theorem 7.7.4]{hsing2015theoretical}, i.e. proving that for any $\varepsilon,\delta>0$ there exists $k>0$ and a finite subset $\{h_1,\dots,h_k\}\subset \Hbb$ such that:
\begin{enumerate}
    \item[(i)] $\inf\limits_{Q\in\mathfrak{D}_{\Fcal}} \mu_Q(E_k^{(\varepsilon)})\geq 1-\delta$, where $E_k= \mathrm{span}(\{h_1,\dots,h_k\})$ and $E_k^{(\varepsilon)} = E_k + B(0,\varepsilon)$.
    \item[(ii)] $\inf\limits_{Q\in\mathfrak{D}_{\Fcal}} \mu_Q(\{h \in \Hbb \;:\; \lvert \langle h,h_j \rangle \rvert \leq r, j=1,\dots,k \}) \geq 1-\delta$ for some $r>0$.
\end{enumerate}

We will use that, by tightness, for all $\gamma>0$ there exists $k\geq 1$ and $r>0$ such that:
\begin{align*}
    \sup_{F\in\Fcal}\sum_{j\geq k} \langle F h_j, h_j \rangle < \gamma,
    \qquad\text{and}\qquad
    \sup_{F\in\Fcal}\frac{1}{r}{\tr(F)}<\gamma.
\end{align*}
For a proof, we refer to \citet[][Example 3.8.13(iv)]{bogachev1998gaussian}.

Let us now prove (i). Let $\{h_i\}_{i\geq 1}$ be a CONS for $\Hbb$ of eigenfunctions for $F$, with $\{\lambda_i\}_{i\geq 1}$ the corresponding set of eigenvalues. Denote by $E_k= \mathrm{span}(\{h_1,\dots,h_k\})$. Let $X$ be an $\Hbb$-valued random variable, $X\sim \mu_Q$ for some $Q\in \mathfrak{D}_{\Fcal}$. Denote by $X^{(k)}$ the projection of $X$ onto $E_k$. Then:
$$
X = X^{(k)} + (X-X^{(k)} ) \in E_k^{(\|X-X^{(k)} \|)},
$$
where $E_k^{(\|X-X^{(k)} \|)} = \{h\in\Hbb \::\: \inf_{u\in E_k} \|h-u\| \leq  \|X-X^{(k)}\|\}$.
Then, note that:
\begin{align*}
    \Pbb ( \| X-X^{(k)} \|^2 > \varepsilon^2 ) \:
    \leq & \: \frac{\Ebb \left[ \sum_{j\geq k} | \langle X, h_j\rangle | ^2 \right] }{\varepsilon^2} \\
    = & \: \frac{1}{\varepsilon^2} \sum_{j\geq k} \langle Q h_j, h_j \rangle \\
    \leq & \: \frac{1}{\varepsilon^2} \sup_{F\in\Fcal}\sum_{j\geq k} \langle F h_j, h_j \rangle \\
\end{align*}
and, by tightness, the upper bound can be made arbitrarily small uniformly, choosing $k$ sufficiently large..
 Therefore, we find that for any $Q\in \mathfrak{D}_{\Fcal}$ and $\varepsilon,\delta>0$, for $k$ large enough: 
$$
\mu_Q(E_k^{(\varepsilon)}) = \Pbb(X \in E_k^{(\varepsilon)} ) = \Pbb ( \| X-X^{(k)} \|^2 \leq \varepsilon^2 ) \geq 1- \delta,
$$
which proves (i).

Next, let us show (ii). 
\begin{align*}
\mu_Q(\{h \in \Hbb \;:\; \lvert \langle h,h_j \rangle \rvert \leq r, j=1,\dots,k \}) \:
    & =  \:\Pbb( \lvert \langle X,h_j \rangle \rvert \leq r, j=1,\dots,k ) \\
    & =  \: 1 - \sum_{j=1}^{k} \Pbb( \lvert \langle X,h_j \rangle \rvert \geq r ) \\
    & \geq  \: 1 - \frac{1}{r^2} \sum_{j\geq 1} \Ebb  \lvert \langle X,h_j \rangle \rvert^2 \\
    & = \: 1 - \frac{\tr (Q)}{r^2}  \\
    &\geq \: 1 -  \sup_{F\in\Fcal}\frac{\tr(F)}{r^2} 
\end{align*}
which again, by tightness, can be made arbitrarily close to $1$ uniformly,  proving (ii). Therefore, the family of Gaussian measures \eqref{eq : family of gaussian measures} is tight, and the corresponding set of covariance operators in \eqref{eq : charac compact set} is hence compact.
\end{proof}



\noindent We can now state and prove the strong law of large numbers in the Bures--Wasserstein space: 

\begin{theorem}\label{thm : consistency of Frechet mean}
Let $\Sigma_1,\ldots,\Sigma_n$ be i.i.d. copies of a random element  $\Sigma$ in $\CovSpace(\Hbb)$ satisfying~\ref{A1}, for $\Hbb$ a separable (possibly infinite dimensional) Hilbert space. Then, with probability 1, any sequence $\{\empbary\}_{n\geq 1}$ of empirical \Frechet means is relatively compact in $(\CovSpace,\Pi)$, and any limit point thereof is a \Frechet mean $\bary$ of $\Sigma$.
\end{theorem}
\begin{proof}[Proof of Theorem~\ref{thm : consistency of Frechet mean}]
Denote by $S_n = \frac{1}{n}\sum_{i=1}^n\Sigma_i$ the arithmetic average. By \citet[][Theorem 12]{masarotto2019procrustes}, we have
\begin{equation}\label{eq : domination}
      \empbary \preceq S_n, \quad \forall\,n\geq 1.
\end{equation}
Let $\Ebb\Sigma_1\in\Bscr_1$ denote the Bochner expectation of $\Sigma_1$ in the trace norm. Under~\ref{A1}, the law of large numbers in Banach space \citep[][Corollary 7.10]{ledoux2013probability} implies $S_n\rightarrow S:=\Ebb \Sigma$ almost surely, and hence tightness of the sequence of averages $\{S_n\}$. Consequently, we have that $\{\empbary\}_{n\geq 1} \subset \mathfrak{D}_{\{S_n\}}$, and hence, by Lemma~\ref{lemma : charac of compact}, this yields the relative compactness of the sequence of empirical barycentres.
Hence, by Prokhorov's theorem, that the sequence of empirical barycentres is uniformly contained in a compact set.
Now, let $\{\empbaryk\}_{k\geq 1}$ be an arbitrary convergent sub-sequence, and it remains to show that the limiting point $\tilde\Xi$ necessarily minimises the population \Frechet functional.
    To do this, let $G\in\CovSpace$ be arbitrary.
    Now fix $\varepsilon>0$ and recall the elementary fact that there exists some $c_{\varepsilon}>0$ such that we have $(a+b)^2 \le c_{\varepsilon}a^2 + (1+\varepsilon)b^2$ for all $a,b\ge 0$.
    Then we can bound:
    \begin{align*}
        \Ebb\left[\Pi^2(\Xi,\Sigma)\right] &= \lim_{k\to\infty}\frac{1}{n_k}\sum_{j=1}^{n_k}\Pi^2(\Xi,\Sigma_j) \\
        &\le \liminf_{k\to\infty}\frac{1}{n_k}\sum_{j=1}^{n_k}\left(c_{\varepsilon}\Pi^2(\Xi,\empbaryk)+(1+\varepsilon)\Pi^2(\empbaryk,\Sigma_j)\right) \\
        &= \liminf_{k\to\infty}\left(c_{\varepsilon}\Pi^2(\Xi,\empbaryk)+(1+\varepsilon)\frac{1}{n_k}\sum_{j=1}^{n_k}\Pi^2(\empbaryk,\Sigma_j)\right) \\
        &= (1+\varepsilon)\liminf_{k\to\infty}\frac{1}{n_k}\sum_{j=1}^{n_k}\Pi^2(\empbaryk,\Sigma_j) \\
        &\le (1+\varepsilon)\lim_{k\to\infty}\frac{1}{n_k}\sum_{j=1}^{n_k}\Pi^2(F,\Sigma_j) \\
        &=(1+\varepsilon)\Ebb\left[\Pi^2(F,\Sigma)\right],
    \end{align*}
    and the conclusion follows by taking $\varepsilon \to 0$.
    
\end{proof}

\begin{remark}\label{remark:SLLN&HeineBorel}
Our Theorem~\ref{thm : consistency of Frechet mean} paralles the celebrated consistency results by \citet{le2017existence}. Note that we could not apply this  to directly derive a Law of Large Numbers, as the latent separable Hilbert space $\Hbb$ is not assumed to be locally compact. Similarly, the results by \citet{evans2020strong} on a Law of Large numbers for \Frechet mean \textit{sets} could not be applied, as the space $\CovSpace$ does not satisfy the Heine-Borel property: closed and bounded sets are not compact. However, it is interesting to notice that the building block of our proof, Lemma~\ref{lemma : charac of compact}, can be seen as a certain form of \emph{ordered Heine-Borel property}, with respect to the partial (Löwner) order.
\end{remark}

An immediate corollary follows from the assumption that $\Sigma$ has a unique barycentre.
\begin{corollary}\label{cor:consistency}
In the framework of Theorem~\ref{thm : consistency of Frechet mean}, if the \Frechet mean $\bary$ of $\Sigma$ is unique, then:
$$
\Pi(\bary,\empbary) \stackrel{\mathrm{a.s.}}{\longrightarrow} 0, \quad \mbox{ as }\, n\rightarrow \infty
$$
for any sequence of empirical \Frechet means $\widehat{\bary}_n$.
\end{corollary}

\begin{proof}[Proof of Corollary~\ref{cor:consistency}]
If the population \Frechet mean $\bary$ is unique, the argument in the proof of Theorem~\ref{thm : consistency of Frechet mean} shows that every sub-sequence $\{\empbaryk\}_{k\geq 1}$ contains a further sub-sequence $\{\empbarykl\}_{l\geq 1}$ converging to $\bary$. Indeed, given an arbitrary sub-sequence  $\{\empbaryk\}_{k\geq 1}$, by Lemma~\ref{lemma : charac of compact} it must admit an almost surely convergent sub-sequence, $\{\empbarykl\}_{l\geq 1}$. In particular, by sequential lower semi-continuity of the \Frechet functional and uniqueness of the population \Frechet mean of $\Sigma$, the limit must be $\bary$. Consequently, by the usual sub-sub-sequence argument, $\{\empbary\}_{n\geq 1}$ must converge to $\bary$.
\end{proof}

\subsection{Fixed-Point Equation and Regularity}\label{subsec:fixedPointandReg}
In this section we investigate certain geometric and analytic properties of Bures-Wasserstein barycentres. In particular, we first show that a Bures-Wasserstein \Frechet mean satisfies a certain fixed-point equation. Furthermore, we provide conditions ensuring that bounded existence of optimal transport maps, which in turn may be used to establish the regularity of barycentres.

\begin{proposition}
\label{prop : fixed point}
Let $\Sigma$ be a random element in $\CovSpace$
with Fr\'echet mean $\bary$. Then $\bary$ satisfies the operator equation

\begin{equation}
    \label{eq : pop bary, fixed point}
    \bary = \Ebb \left[  \left(\bary ^{\sfrac{1}{2}} \Sigma \bary^{\sfrac{1}{2}} \right)^{\sfrac{1}{2}}\right].
\end{equation}
where the expectation is taken in $(\Bscr_1,\|\cdot\|_1))$. In particular, further assuming~\ref{B1}, the we have the equality:
\begin{equation}\label{eq : pop bary, fixed point on map}
    \Ebb T_{\bary}^{\Sigma} = \Id,
\end{equation}
where the expectation in taken in $(\Tan_{\bary},\langle\cdot,\cdot\rangle_{\bary})$. 
\end{proposition}
The analogue version of this statement on the empirical level asserts that any barycentre of covariances $\Sigma_1,\dots,\Sigma_n$ necessarily satisfies:
\begin{equation}
    \label{eq : sample fix point equation}
\empbary=\frac{1}{n} \sum_{i=1}^{n}\left(\empbary^{\sfrac{1}{2}} \Sigma_{i} \empbary^{\sfrac{1}{2}}\right)^{\sfrac{1}{2}}.
\end{equation}

Concretely, the fixed point equation states that $\bary$ is on average preserved by random displacements $T_{\bary}^{\Sigma}$; or, equivalently that the push-forward of $P$ by the logarithmic embedding is centered in the tangent space $\Tan_\bary$.
In fact, equation \eqref{eq : pop bary, fixed point on map} relates the realisations of $\Sigma$ to their \Frechet mean $\bary$ at the level of the tangent space at $\bary$, essentially stating that they can be represented ``zero mean perturbations" of that mean, at least when restricted on the range of $\bary^{1/2}$.  The equation is intriguing because it relates a \Frechet mean with a Bochner mean. This connection has been used in finite dimensions to obtain a central limit theorem. For reasons that will become clear later, there are generic obstructions to this approach in infinite-dimensions, and in this case it is Equation \eqref{eq : pop bary, fixed point} that will be exploited instead. See Section~\ref{subsec : clt}.

\begin{proof}[Proof of Proposition~\ref{prop : fixed point}]
Let $\{P_k\}_{k\geq 1}$ be a sequence of finite-dimensional priojections on $\Hbb$ converging strongly to the identity. Consider the random (finite-dimensional) covariance $ P_k \Sigma P_k$. By \cite{alvarez2016fixed}, the corresponding barycentre $\bary_{(k)}$ satisfies the equation:
$$
\Sigma^{(k)} = \Ebb\left[ \left(\bary_{(k)}^{\sfrac{1}{2}} P_k \Sigma P_k \bary_{(k)}^{\sfrac{1}{2}}\right)\right].
$$
Note that for all $k\geq 1$ we have $\bary_{(k)}\preceq \Ebb[P_k\Sigma P_k] \preceq \Ebb[\Sigma]$, so that the sequence $\{\bary_{(k)}\}_{k\geq 1}$ is tight by Lemma~\ref{lemma : charac of compact}. By the usual argument, we see that $\bary_{(k)}\to\bary$ as $k\to\infty$ in trace norm, so that using continuity we obtain that \eqref{eq : pop bary, fixed point} holds for $\bary$.

\medskip

Due to the instability of optimal transport maps (see Proposition~\ref{prop : counterexample}) the second fixed-point equation cannot be deduced from the corresponding finite-dimensional result, and requires a more direct approach. Consider $\mu \equiv \Ncal(0,\bary)$ and the random $\nu \equiv  \Ncal(0,\Sigma)$. Denote by $T = T_{\bary}^{\Sigma}$ the (random) optimal map, which exists by assumption. Let $\Psi(F)=\mathbb{E} [\Pi(\Sigma, F)^{2}]$ the \Frechet functional. We follow similar steps as in the finite-dimensional case (e.g. \citet[Theorem 3.10]{alvarez2018wide}), but now adjusting for the possible infinite dimensionality of the underlying space.
Let us define $\bar{t}:= \Ebb[T_{\bary}^{\Sigma}] \in \Tan_{\bary}$
Now, note that:
\begin{align*}
    \Ebb\left[\Pi^2(\Sigma,\bary)\right]
    \: =& \:  \Ebb\left[\|T_{\bary}^\Sigma - \Id\|_{\bary}^2\right]
    \\ =& \:   \Ebb\left[\|T_{\bary}^\Sigma - \bar{t}\|^2_{{\bary}}\right] + \Ebb\left[\|\bar{t} -\Id\|^2_{{\bary}}\right]
    \\ \geq & \:   \Ebb\left[\|T_{\bary}^\Sigma - \bar{t}\|^2_{{\bary}}\right]
    \\ =& \: \Ebb\left[\Pi^2(\Sigma,\bar t\bary \bar t)\right]
\end{align*}
and by uniqueness of the minimiser we get $\bary = \bar t \bary \bar t$
proving that necessarily $\bar t =\Id\in\Tan_M$, and hence the fixed point equation.

\end{proof}

When $\Hbb$ is finite dimensional, the barycentre of an almost surely regular covariance is regular as well \citep{agueh2011barycenters} and is characterised by \eqref{eq : pop bary, fixed point}. 
In the general (possibly infinite dimensional) case, a solution to the fixed point equation is far from being unique \citep{minh2022entropic}, and to deduce that a solution to \eqref{eq : pop bary, fixed point} is also a \Frechet mean, one additionally needs to require that $\bary\succ 0$ \citep{masarotto2019procrustes} naturally leading to the question of \textit{regularity} of barycenters.
    
\medskip

 In finite dimensions,  ~\ref{A1} and~\ref{A2}  imply not only  uniqueness but also regularity of the  population \Frechet mean. This need not be true in infinite dimensions, without additional conditions, as shown in \cite{zemel2023non}, which establishes an almost-surely regular random covariance with non-regular barycentre. Nevertheless,  Proposition~\ref{prop : bary positive} establishes a sufficient condition to that effect. 


\begin{proposition}\label{prop : bary positive}
 Let $\Sigma \in \CovSpace$ be  a random covariance operator  satisfying: 
   	\begin{equation}
 \label{eq_makingTheBaryPositive}
    \Pbb\Big\{\Sigma \succ 0\,\,\&\,\, T_{\bary}^{\Sigma}\text{ exists and is invertible}\Big\} >0,
   	\end{equation}
Then its (unique) \Frechet mean $\bary\in\CovSpace$ satisfies $\bary\succ0$.
\end{proposition}

\begin{proof}[Proof of Proposition~\ref{prop : bary positive}]
Under \eqref{eq_makingTheBaryPositive}, it is easy to see that there exists $\Omega'\subset \Omega$ such that for every $\omega\in\Omega'$ we have that $\Sigma(\omega)\succ 0$, $T(\omega) := T_{\bary}^{\Sigma(\omega)}\succ 0$, and $\|T(\omega)h\| \geq c(\omega) \|h\|$ for all $h\in\Hbb$ and some $c(\omega)>0$. Therefore, $T(\omega)$ has a closed range, and we may conclude that $\range(T(\omega))=\Hbb$. Hence, for any $y\in\Hbb$, $y\neq 0$, there exists $x\in\Hbb$, $x\neq 0$, such that $T(\omega)x=y$, from which it follows that
$$
\langle \bary y,y\rangle = \langle \bary T(\omega)x,T(\omega)x\rangle = \langle T(\omega)\bary T(\omega)x,x\rangle = \langle \Sigma(\omega) x,x\rangle >0
$$
for all $\omega \in \Omega'$. Since the left hand side is deterministic and $y\in \Hbb$ was arbitrary, we conclude that $\bary\succ0$.
\end{proof}

The almost sure bounded invertibility of the optimal maps $T_{\bary}^\Sigma$ posited in Equation \eqref{eq_makingTheBaryPositive} can be secured by first principles by assuming~\ref{B2}, as stated in the following lemma.

\begin{lemma}\label{lemma:boundbothmaps}
 Let $\Sigma \in \CovSpace$ be  a random covariance operator  satisfying~\ref{B2}, and write $\bary$ for a barycentre. Then:
 $$
 \|T_{\bary}^{\Sigma}\|_{\infty} + \|T_{\Sigma}^{\bary}\|_{\infty} < \infty,\qquad \text{almost surely.}
 $$
 \end{lemma}

Before presenting the proof, we state an operator-theoretic result that we will need and use in several occasions.
\begin{lemma}
\label{lemma : A<B}
Let $A,B$ be bounded operators on a Hilbert space $\Hbb$. If $0 \preceq A \preceq B$, then:
\begin{enumerate}

    \item[(1)] $ C^* A C \preceq C^{*} B C$, for any bounded operator $C$.

    \item[(2)] $A^{\sfrac{1}{2}} \preceq B^{\sfrac{1}{2}}$.

    \item[(3)] There exists a bounded operator $G$ such that $A^{\sfrac{1}{2}}=B^{\sfrac{1}{2}} G$ and $\operatorname{ker} G^{*} \supseteq \operatorname{ker} B$.
\end{enumerate}
\end{lemma}
The proof of (1) is trivial; for (2), see the (only) Theorem in \citet{pedersen1972some}; for (3) see Corollaries 2(b) and 1(a) in \citet{baker1970covariance}. 
We may now proceed with the proof of Lemma~\ref{lemma:boundbothmaps}

\begin{proof}[Proof of Lemma~\ref{lemma:boundbothmaps}]

Note that $\Sigma$ admits a unique (population) barycenter $\bary$ by Theorem~\ref{thm : existence and uniqueness}. By \eqref{eq : domination}, using Theorem~\ref{thm : consistency of Frechet mean}, and the usual law of large numbers,  it is easy to see that  for any $h\in\Hbb$ :
$$
\langle \bary h, h\rangle = \lim_{n\to \infty} \langle \empbary h , h\rangle \leq \lim_{n\to \infty} \langle S_n h , h\rangle = \langle S h , h\rangle,\qquad \text{almost surely,}
$$
which establishes that $\bary \preceq \Ebb[\Sigma]$.
Next, observe that if $\Sigma'$ is a copy of $\Sigma$, then assumption~\ref{B2} entails their almost sure equivalence:
$$
c(\Sigma) \frac{\Sigma'}{C(\Sigma')} \preceq \Sigma
 \preceq C(\Sigma)\frac{\Sigma'}{c(\Sigma')},\qquad \text{almost surely,}
$$
and in particular:
$$
\bary \preceq \Ebb[\Sigma] \preceq \Ebb[C(\Sigma)]\frac{\Sigma'}{c(\Sigma')},\qquad \text{almost surely.}
$$
 By Lemma~\ref{lemma : A<B}, multiplying left and right by ${\Sigma'}^{\sfrac{1}{2}}$ and taking square roots yields that:
$$
({\Sigma'}^{\sfrac{1}{2}} \bary  {\Sigma'}^{\sfrac{1}{2}})^{\sfrac{1}{2}} \preceq \left(\frac{\Ebb\left[C(\Sigma) \right]}{c(\Sigma')}\right)^{\sfrac{1}{2}}\cdot \Sigma' ,\qquad \text{almost surely.}
$$
Hence, again by  Lemma~\ref{lemma : A<B}, there exists a bounded operator $G=G(\Sigma')$ with $\|G\|_{\infty}\leq 1$ such that:
$$
  {\Sigma'}^{-\sfrac{1}{2}}({\Sigma'}^{\sfrac{1}{2}} \bary  {\Sigma'}^{\sfrac{1}{2}})^{\sfrac{1}{2}}{\Sigma'}^{-\sfrac{1}{2}} =  \left(\frac{\Ebb\left[C(\Sigma) \right]}{c({\Sigma'})}\right)^{\sfrac{1}{2}}GG^*,\qquad \text{almost surely,}
  $$
and in particular, for almost every $\Sigma\in \CovSpace$, the optimal map $T_{\Sigma}^{\bary}$ exists and is a bounded, linear operator on $\Hbb$.

\medskip

It remains to prove that the optimal map in the converse direction is almost surely bounded as well. Similar to the above, we see that:
$$
\frac{c(\Sigma)^{\sfrac{1}{2}}  }{C(\Sigma')^{\sfrac{1}{2}}}\left(\bary^{\sfrac{1}{2}}\Sigma'\bary^{\sfrac{1}{2}}\right)^{\sfrac{1}{2}}
\preceq
\left(\bary^{\sfrac{1}{2}}\Sigma\bary^{\sfrac{1}{2}}\right)^{\sfrac{1}{2}} 
$$
In particular, integrating out $\Sigma$, the fixed point equation \eqref{eqn:pop-fixed-pt} gives:
$$
\frac{\Ebb[c(\Sigma)^{\sfrac{1}{2}}]}{C(\Sigma')^{\sfrac{1}{2}}}\left(\bary^{\sfrac{1}{2}}\Sigma'\bary^{\sfrac{1}{2}}\right)^{\sfrac{1}{2}} \preceq  \bary ,\qquad \text{almost surely,}
$$
and again by Lemma~\ref{lemma : A<B}  there exists
bounded operator $L=L(\Sigma)$  with $\|L\|_{\infty}\leq 1$ such that:
$$
  {\bary}^{-\sfrac{1}{2}}({\bary}^{\sfrac{1}{2}} \Sigma'  {\bary}^{\sfrac{1}{2}})^{\sfrac{1}{2}}{\bary}^{-\sfrac{1}{2}} = \frac{C(\Sigma')^{\sfrac{1}{2}}}{\Ebb[c(\Sigma)^{\sfrac{1}{2}}]} LL^*,\qquad \text{almost surely,}
  $$
similarly to the first part of the proof, which finally entails that for almost every $\Sigma\in \CovSpace$ the optimal map $T_{\bary}^{\Sigma}$ exists and is a bounded, linear operator on $\Hbb$.
\end{proof}

We have in particular shown that in the presence of assumptions~\ref{B2}, the fixed-point equation fully characterises infinite-dimensional Bures-Wasserstein barycenters.

\begin{corollary}\label{cor:FPE}
    Let $\Sigma\in \CovSpace$ be a random covariance operator satisfying~\ref{B2}. Then and $\bary\in\CovSpace$ is the barycenter of $\Sigma$ if and only if
    \begin{equation}\label{eqn:pop-fixed-pt}
	\bary\succ 0 \qquad \text{and}\qquad \Ebb[T_{\bary}^{\Sigma}-I] = 0,
\end{equation}
where the expectation is a Bochner integral in the Hilbert space $\Tan_M$.
\end{corollary}

\begin{proof}
For the first direction, let $M$ satisfy \eqref{eqn:pop-fixed-pt} and let us prove that $M$ need be the barycenter of $P$, i.e.\ $M=\bary$. Let us write $\mu_M$ for the centred gaussian measure $\Ncal(0,M)$.
As in the proof of \cite[][Theorem 14]{masarotto2019procrustes} define the convex potential $\phi(x) = \langle t_{M}^{\Sigma}x, x \rangle$ and denote by $\phi^*(y) = \sup_{x\in\Hbb}\langle x,y\rangle - \phi(x)$ its Legendre transform. Weak Kantorovich duality then yields that for any $ M\in\CovSpace$ the following lower bound holds:
\begin{align*}
    \frac{1}{2}\Pi( M, \Sigma)^2  \geq \int_{\Hbb}\left(\frac{1}{2}\|x\|^2 - \phi(x)\right) \diff \mu_{ M} + \int_{\Hbb}\left(\frac{1}{2}\|y\|^2 - \phi^*(y)\right) \diff \mu_{\Sigma}, \qquad\text{almost surely,}
\end{align*}
while strong Kantorovich duality gives the following equality:
\begin{align*}
    \frac{1}{2}\Pi(\bary, \Sigma)^2  = \int_{\Hbb}\left(\frac{1}{2}\|x\|^2 - \phi(x)\right) \diff \mu_{ \bary} + \int_{\Hbb}\left(\frac{1}{2}\|y\|^2 - \phi^*(y)\right) \diff \mu_{\Sigma}, \qquad \text{almost surely.}
\end{align*}
Thus, taking expectations, since $\Ebb\left[ \phi(x) \right] = {\|x\|^2}/{2}$, we obtain via Fubini's theorem that:
\begin{align*}
\Ebb[\Pi( \bary, \Sigma)^2 ]
\:=&\:
2\Ebb\left[\int_{\Hbb}\left(\frac{1}{2}\|y\|^2 - \phi^*(y)\right) \diff \mu_{\Sigma}\right]
\\ \leq&\:  
\Ebb\left[\Pi( M, \Sigma)^2 \right], \qquad \forall\:  M \in \CovSpace.
\end{align*}

For the second direction, first note that Proposition~\ref{lemma:boundbothmaps} and Proposition~\ref{prop : bary positive} entail that if $\bary$ is the barycentre, then it satisfies $\bary\succ 0$, so that the optimal map $T_{\bary}^{\Sigma}$ exists on a dense subspace of $\mu_{\bary}$-measure $1$. Note that~\ref{B2} implies that $\Ebb\left[|T_{\bary}^\Sigma - \Id\|^2_{{\bary}}\right]<\infty$, hence proving the existence of the Bochner integral $ \Ebb\left[(T_{\bary}^{\Sigma}-\Id)\right] \in \Tan_{\bary}$. Finally, repeating the steps in the proof of (the second part of) Proposition~\ref{prop : fixed point}, we may conclude that $\bary$ satisfies the fixed point equation.
\end{proof}

\subsection{Convergence of Optimal Maps}\label{subsec : estimation maps}

Further to the convergence of an empirical \Frechet mean $\empbary$ to its population counterpart $\bary$, one may be interested in the convergence properties of the empirical optimal maps $T_{\empbary}^{\Sigma_i}$ to their population counterparts $T_{\bary}^{\Sigma_i}$. In addition to being natural objects, these maps have important probabilistic and statistical roles, in the construction of optimal multicouplings and tangent-space principal component analyses, respectively (see, e.g., \citet{zemel2019frechet}). In the finite-dimensional setting $\Hbb= \Rbb^d$, \citet[][Proposition 6]{zemel2019frechet} grants convergence of the maps in the operator norm, provided certain (mild) regularity conditions hold. 
 Unfortunately, the following negative result illustrates that one cannot generally expect the same mode of convergence in infinite-dimensions.

\begin{proposition}[Instability of optimal maps when $\mathrm{dim}(\Hbb)=\infty$]\label{prop : counterexample} For any $\bary\succ 0$ on an infinite dimensional separable Hilbert space $\Hbb$, there exists $F\succ 0$, a sequence $\empbary\succ 0$, and $\varepsilon>0$
such that  
$$\Pi(\bary,\empbary)\stackrel{n\to\infty}{\longrightarrow}0, $$
while
$$\|T_{\empbary}^{F}-T_{\bary}^{F}\|_{\infty}>\varepsilon,\quad \forall\, n\geq 1.$$
\end{proposition}
\begin{proof}
By the spectral theorem, $\bary=\sum_{j\geq 1}\lambda^2_j e_j\otimes e_j$ for $\{e_k\}_{k\geq 1}$ a complete orthonormal system of $\Hbb$ and a square-summable positive sequence $\{\lambda_k\}_{k\geq 1}$. Define $E_n=\sum_{j\leq n}\lambda^2_j e_j\otimes e_j$ to be a rank-$n$ truncation of $\bary$. For $a,b>0$ such that $a^2+b^2=1$ let $\varepsilon= |1-b^{-1}|>0$ and define
$$\empbary = a^2 E_n+ b^2 \bary.$$
Writing $\bary-\empbary=(1-b^2)\bary-a^2E_n=a^2(\bary-E_n)$, we observe that
$$\Pi(\empbary,\bary)\to 0,\quad \mbox{ as }n\to\infty.$$
\noindent Now note that $\bary\succ 0$ and $\empbary\succ 0$, and they commute.
It follows that $T_{\empbary}^{\bary}= \bary^{\sfrac{1}{2}}\empbary^{-\sfrac{1}{2}}$, which reduces to
\begin{align*}
 \bary^{\sfrac{1}{2}}\empbary^{-\sfrac{1}{2}} &=
\left(\sum_{j\geq 1}\lambda_j e_j\otimes e_j\right)\left(\sum_{j\geq 1}\frac{1}{\lambda_j\sqrt{a^2\mathbf{1}\{j\leq n\}+b^2}} e_j\otimes e_j\right)\\
&=\sum_{j\leq n}e_j\otimes e_j+\frac{1}{b}\sum_{j>n} e_j\otimes e_j.
\end{align*}
Thus, for any $n\geq 1$ we have
$$\|(T_{\bary}^{\bary}-T_{\empbary}^{\bary})e_{n+1} \|=\|(\Id-T_{\empbary}^{\bary}) e_{n+1}\| =\left\|e_{n+1}-b^{-1}e_{n+1}\right\|=|1-b^{-1}|=\varepsilon>0.$$
Setting $F=\bary$ completes the proof.
\end{proof}

 We nevertheless show in our next Theorem that, under rather mild conditions, pointwise convergence (also known as strong operator convergence) \emph{can} be obtained:

\begin{theorem}\label{thm : convergence of optimal maps}
Let $\Sigma_1,\ldots,\Sigma_n$ be i.i.d.\ copies of a random element  $\Sigma$ in $\CovSpace(\Hbb)$ for $\Hbb$ a separable (possibly infinite dimensional) Hilbert space. Assume that $\Sigma$ admits a unique and regular population \Frechet mean $\bary\succ 0$ and that $\|T_{\bary}^{\Sigma}\|<\infty$ almost surely. Then, the following pointwise convergence result holds:
$$
\forall \: h\in\Hbb
\:,\quad \: \left\|T_{\bary}^{\Sigma_1} h -  T_{\empbary}^{\Sigma_1} h \right\| \stackrel{\mathrm{a.s.}}{\longrightarrow} 0\quad\text{ as } n\rightarrow \infty.
$$
\end{theorem}
\begin{proof}[Proof of Theorem~\ref{thm : convergence of optimal maps}] For notational simplicity, we write  $T = T_{\bary}^{\Sigma_1}$ and $ T_{n} = T_{\empbary}^{\Sigma_1}$ for the optimal maps from the population and sample barycentres to $\Sigma_1$, respectively. The proof will be broken down into four steps.

\noindent{\textit{Step 1: weak convergence of optimal plans}}.
Consider the Gaussian measures $\mu \equiv \Ncal(0,\bary)$, $\mu_n \equiv \Ncal(0,\empbary)$. By Theorem~\ref{thm : consistency of Frechet mean}, $\Pi(\bary, \empbary)\rightarrow 0$ almost surely, and equivalently $W_2(\mu, \mu_n) \rightarrow 0$. To the random element $\Sigma_1$ corresponds a (random) measure $\nu_1 \equiv \Ncal(0,\Sigma_1)$.
Let  $\pi$ be the \textit{unique} optimal coupling between $\mu$ and $\nu_1$, which is supported on the graph of the optimal map $T$. Note that uniqueness is ensured by our asumption that $\bary\succ 0$. Let $ \pi_n$ be an optimal coupling between $\mu_n$ and $\nu_1$ which is supported on the graph of the optimal map $T_{n}$. Now note that uniqueness of $\pi$ and the stability of optimal transportation imply that $\pi_n$ converges weakly to $\pi$ (see \citet[][Theorem 3]{schachermayer2009characterization}).

\medskip 

\noindent{\textit{Step 2: uniform boundedness}}.
Consider the set
$$
C = \{(h_1,h_2)\in \Hbb^2 \::\: \|h_1\| \leq 1 \text{ and } \|h_2\| \geq 2\|T\|_{\infty}\}
$$
which is closed in the product topology on $\Hbb^2$. By the Portmanteau lemma,
\begin{equation} \label{eq : proof portmanteau}
0 \leq \liminf_{n\rightarrow\infty} \pi_n(C) \leq \limsup_{n\rightarrow\infty} \pi_n(C) \leq \pi(C).
\end{equation}
Since $\pi$ is supported on the graph of $T$, we have
$$
\pi(C) = \mu \left( \{h\in\Hbb\::\: \|h\|\leq 1 \text{ and }\|Th\|\geq 2\|T\|_{\infty}\} \right) = 0,
$$
as the set appearing on the right-hand-side is empty. By \eqref{eq : proof portmanteau}, since $\pi_n$ is supported on the graph of $T_n$, we obtain:
$$\lim_{n\rightarrow\infty}   \mu_n\left( \{h\in\Hbb\::\: \|h\|\leq 1, \: \|T_n h\|\geq 2\|T\|_{\infty}\} \right) = 0.
$$
In particular, for every $k\geq 1$ we may find $n_k$ such that 
$$
\mu_{n_k}\left( \{ h\in\Hbb\::\: \|h\|\leq 1, \: \|T_{n_k} h\|\geq 2\|T\|_{\infty}\} \right) < 2^{-k}.
$$
Therefore, taking a sequence of random variables $X_{n_k} \sim \mu_{n_k}$:
$$
\sum_{k\geq 1} \Pbb(\|X_{n_k}\|\leq 1, \: \|T_{n_k}X_{n_k}\|\geq 2\|T\|_{\infty}) 
= \sum_{k\geq 1} \mu_{n_k}\left( \{h\in\Hbb\::\: \|h\|\leq 1, \: \|T_{n_k} h\|\geq 2\|T\|_{\infty}\} \right) < \infty.
$$
By the (first) Borel-Cantelli lemma, we obtain that for large enough $k$:
$$
 \{ h\in\Hbb\::\: \|h\|> 1 \text{ or } \|T_{n_k} h\| \geq 2\|T\|_{\infty}\}, \quad \mu_{n_k} \text{ a.s.}
$$
or equivalently:
$$
 \{ h\in\Hbb\::\:\text{ if } \|h\|\leq 1 \text{ then } \|T_{n_k} h\| < 2\|T\|_{\infty}\}, \quad \mu_{n_k} \text{ a.s.}.
$$
Therefore, we have:
$$
\sup_{h\in \mathrm{supp}(\mu_{n_k)}, \|h\|\leq 1} \|T_{n_k}h\|_{\infty}\leq 2\|T\|_{\infty}
$$
and since each $T_{n_k}$, is linear (in particular, homogeneous), one can always
rescale an arbitrary $h \in \Hbb$ to have norm bounded above by $1$, hence:
$$
\sup_{h\in \mathrm{supp}(\mu_{n_k})} \frac{\|T_{n_k}h \|}{\|h\|}< 2\|T\|_{\infty}
$$
In particular, by continuity of the $T_{n_k}$, the bound above holds on the closure of the support of $\mu_{n_k}$. 
However, $T_{n_k}$ is null outside of such a subspace (see  \citet{masarotto2022transportation}), and thus the bound holds on  $\Hbb$: $\|T_{n_{k}}\|_{\infty}< 2 \|T\|_{\infty}$ for $k$ sufficiently large. 

Now we argue by contradiction. Suppose that the sequence $T_n$  satisfy the above bound eventually in $n$. This would imply that there exists a sub-sequence $\{T_{n_l}\}_{l\geq 1}$ such that $\| T_{n_l}\|_{\infty}\geq 2\|T\|_{\infty}$ for every $l\geq 1$. However, reproducing the arguments presented above, there necessarily exists a further sub-sequence $\{T_{n_{l_k}}\}_{k\geq 1}$ that needs to satisfy the required bound. 

\medskip

\noindent{\textit{Step 3: optimal maps are close}}. For $h\in \Hbb$ arbitrary, let $B(h,\delta)$ be the open ball of radius $\delta$ centred at $h$. By continuity of $T$, for every $\varepsilon>0$ there exits $\delta = \delta(\varepsilon)>0$ such that $h' \in B(h,\delta) \:\Rightarrow\: Th' \in B(Th,\varepsilon)$. 
Take any sequence $\varepsilon_k\rightarrow 0$ and the corresponding $\delta_k\rightarrow 0$.
For any $k\geq 1$:
\begin{align*}
\pi\big(B(h,\delta_k) \times B(Th,\varepsilon_k) \big) 
=&
\pi\big( \{ (h', Th') \::\:h' \in B(h,\delta_k), T h' \in B(Th,\varepsilon_k)  \} \big) 
\\=&
\mu\big(B(h,\delta)\big)>0.
\end{align*}
By the Portmanteau lemma, weak convergence of $\pi_n$ to $\pi$ implies that, for any open set $A$, $\liminf_n \pi_n(A)\geq \pi(A).$ In particular,
$$
\liminf_n \pi_n \left(B(h,\delta_k) \times B(Th,\varepsilon_k) \right)  >0.
$$
Since $\pi_n$ is supported on the graph of $T_n$, this implies that for any $k\geq 1$, there exists $n_k\geq 1$ and $h_{n_k}\in \overline{\mathrm{supp}(\mu_{n,k}})$ such that $\|h-h_{n_k}\|<\delta_k$ and $\| T_{n_k}h_{n_k} - Th\| < \varepsilon_k.$

\medskip

\noindent{\textit{Step 4: pointwise convergence}}. By the triangle inequality,
$$
\|Th - T_{n_k}h\| \: \leq \: \|Th - T_{n_k}h_{n_k}\| + \|T_{n_k}(h_{n_k} - h)\| \: \leq \: \varepsilon_k + \|T_{n_k}\|_{\infty}\delta_k.
$$
To conclude that the sub-sequence of operators $\{T_{n_k}\}_{k\geq 1}$ converges pointwise to $T$ it suffices to have an (integrable) uniform (in $n$) bound on $\{\|T_{n}h\|\}_{n\geq 1}$ for all $h\in \overline{\mathrm{supp}(\mu_{n}})$. But this we obtained in \textit{Step 2}.
Therefore, we can conclude that there exists a sub-sequence of operators $\{ T_{n_k}\}_{k\geq 1}$ converging to $T$ pointwise. In fact, it is straightforward to show that any sub-sequence of operators admits a further sub-sequence converging pointwise to $T$, and that thus the entire sequence $T_n$ necessarily converges to $T$ pointwise.


\end{proof}

The pointwise result also implies convergence in probability on the tangent space $\mathcal{T}(\bary)$ at the population \Frechet mean, as stated in the following corollary.
\begin{corollary}\label{cor : convergence of maps in the tangent space}
Under the assumptions of Theorem~\ref{thm : convergence of optimal maps} and with the same notation:
$$
\left\| T_{\empbary}^{\Sigma_1} - T_{\bary}^{\Sigma_1} \right\|_{\Tan(\bary)} \overset{\Pbb}{\longrightarrow} \:0.
$$
\end{corollary}

\begin{proof}[Proof of Corollary~\ref{cor : convergence of maps in the tangent space}]
Let us write $T, T_{n} $ for $T_{\bary}^{\Sigma_1}$ and $T_{\empbary}^{\Sigma_1}$, respectively. By definition of the tangent space norm,
\begin{align}
   \left\| T_{\empbary}^{\Sigma_1} - T_{\bary}^{\Sigma_1} \right\|^2_{\Tan(\bary)}= \| \big(T - T_n\big)\bary^{\sfrac{1}{2}}  \|_{2} ^2
    = & \:
    \sum_{k\geq 1} \| \big(T - T_n\big)\bary^{\sfrac{1}{2}} e_k\|^2
    \nonumber \\ = & \: 
    \sum_{k\geq 1} \lambda_k \| \big(T - T_n\big) e_k \|^2
    \nonumber \\ \leq & \: 
    \lambda_1 \sum_{k=1}^N \| \big(T - T_n\big) e_k \|^2 
    +
    5 \| T \|_{\infty}^2  \sum_{k > N} \lambda_k
    \label{eqCovEst_FINAL2}
\end{align}
where we have used the uniform bound obtained in Step 2 in the proof of Theorem~\ref{thm : convergence of optimal maps}.
Now consider the second term in \eqref{eqCovEst_FINAL2}. For every $\varepsilon>0$, there exists $N = N(\varepsilon)$ such that surely:
$
\sum_{k > N} \lambda_k < \varepsilon.
$
Furthermore, $\|T\|_{\infty} <\infty$ by assumption. Hence,
Therefore, choosing $N=N(\varepsilon/\|T\|_{\infty})$, we get that 
for every $\varepsilon>0$:
\begin{equation} \label{eq:smallprob}
    \Pbb \Big(\|T\|_{\infty}^2 \sum_{k > N} \lambda_k < \varepsilon \Big) \geq 1-\delta.
\end{equation}
We now turn to the first term in \eqref{eqCovEst_FINAL2}. By Theorem~\ref{thm : convergence of optimal maps} we know that $T - T_n$ converges pointwise almost surely to $0$ on the range of $\bary^{\sfrac{1}{2}}$. Therefore, $T - T_n$ almost surely converges to zero on any finite subset of the eigenvalues of $\bary$. In particular, $n$ sufficiently large, the term  $\sum_{k = 1}^N \| \big(T - T_n\big) e_k \|^2 $ may be rendered arbitrarily small with high probability.  
\end{proof}


 \subsection{Central Limit Theorem}\label{subsec : clt}

In this section we show that the distance 
$\Pi(\bary,\empbary)$ between the population and sample barycentre is $\sqrt{n}$-tight. Specifically, we will show that the scaled difference $\sqrt{n}(\bary^{\sfrac{1}{2}}-\empbary^{\sfrac{1}{2}})$ converges to a centred Gaussian random element in the space of Hilbert-Schmidt operators.

In finite dimensions ($\Hbb=\mathbb{R}^d$) several authors established results in this direction. {\citet{kroshnin2021statistical} obtain a central limit theorem in the finite-dimensional case, via the transport maps involved in  \eqref{eq : pop bary, fixed point on map} and their empirical versions. Necessary to the success of this approach is the \Frechet differentiability of the functional $\bary \mapsto T_{\bary}^{F}$ when $\Hbb=\mathbb{R}^d$ and $\bary\succ 0$, combined with finer properties of the associated differential (Lemma A.3 in \citet{kroshnin2021statistical}). However, such a functional will generally fail to be differentiable at all for general $\Hbb$, much less posses a differential with the analogous properties. In fact, the functional will typically fail to even be continuous: see the counterexample in Proposition~\ref{prop : counterexample}.
{In light of this, we will circumvent the direct use of Equation \eqref{eq : pop bary, fixed point on map} which requires going through the (intangible) regularity of $\bary\mapsto T_{\bary}^F$. Instead, we will take a more indirect $Z$-estimation approach, making use of Equation \eqref{eq : pop bary, fixed point}, and going through the (tangible) regularity of a different functional. 

Again in the finite dimensional case, \citet{le2022fast} show the rate of convergence building on the notion of extendable geodesics: for $\lambda_{\text {in }}, \lambda_{\text {out }}>0$, a constant-speed geodesic $\gamma:[0,1] \rightarrow \CovSpace$ is $\left(\lambda_{\text {in }}, \lambda_{\text {out }}\right)$-extendible if there exists a constantspeed geodesic $\gamma^{+}:\left[-\lambda_{\text {in }}, 1+\lambda_{\text {out }}\right] \rightarrow \CovSpace$ such that $\gamma$ is the restriction of $\gamma^{+}$to $[0,1]$. 
In their work  \citet{le2022fast} guarantee that geodesics are bi-extendable by bounding the minimal eigenvalue of $\Sigma$ away from zero; this, however, is impossible to transfer to the case of general $\Hbb$, as would contradict traceability in infinite dimensions. Nevertheless, we show that this program can still be executed in the presence of a stronger variant of assumption~\ref{B2}.

\begin{proposition}\label{prop:biextend}
    Let $\Sigma$ be random element of $\CovSpace(\Hbb)$. Fix $0<\kappa_0<\kappa_1$ and $R\in\CovSpacePos$, and assume that:
    \begin{equation}\label{eq:strongdom}
    \kappa_0 R \preceq \Sigma \preceq \kappa_1 R,\qquad\textnormal{almost surely.}
    \end{equation}
    Then, if $\kappa = (\kappa_1/ \kappa_0)^{\sfrac{1}{2}}$ satisfies $\kappa - \kappa^{-1}<1$, we have that:
    $$
    \Ebb\left[\Pi^2(\bary,\empbary)\right]
    \leq \frac{4\Ebb[\Pi^2(\bary,\Sigma)]}{(1 - \kappa + \kappa^{-1})^2\:n},\qquad n\geq 1.
    $$
\end{proposition}

\begin{proof}[Proof of Proposition~\ref{prop:biextend}]
Note that \eqref{eq:strongdom} implies, by the proof of Lemma~\ref{lemma:boundbothmaps}, that  
$$
\|T_{\bary}^{\Sigma}\|_{\infty} < \kappa
\qquad \text{and}\qquad
\|T_{\Sigma}^{\bary}\|_{\infty} < \kappa
$$
so that almost every $\Ncal(0,\Sigma)$ is obtained by pushforward of $\Ncal(0,\bary)$ by the gradient of an $\kappa^{-1}$-strongly convex and $\kappa$-smooth function. Hence we can conclude by \citet[][Corollary 16]{le2022fast}.
\end{proof}

 The rather stringent assumptions leading to the parametric rate of convergence via \citet[][Corollary 16]{le2022fast} turn out to be superfluous in the Bures-Wasserstein context. In fact, we will show that under strictly weaker conditions we can get a \textit{stronger} result in the form of a central limit theorem.
That is, we prove under~\ref{B2} that the scaled difference $\sqrt{n}(\bary^{\sfrac{1}{2}}-\empbary^{\sfrac{1}{2}})$ converges in law to a centred Gaussian random element in the space of Hilbert-Schmidt operators.

Let us now describe the high-level heuristic of our proof strategy.} Recall that the population and sample barycentres respectively satisfy \eqref{eq : pop bary, fixed point} and \eqref{eq : sample fix point equation}.
Consider the functional:
$$
\phi_n\::\: 
 F\mapsto \frac{1}{n}\sum_{j=1}^n(F^{\sfrac{1}{2}}\Sigma_j F^{\sfrac{1}{2}})^{\sfrac{1}{2}} - F.
$$
Assuming that we can Taylor expand, we hope to be able to write:
\begin{align*}
    \phi_n(\bary) = \phi_n(\empbary) + \der_{\empbary}\phi_n (\bary- \empbary) + o(\bary- \empbary),\qquad\mbox{almost  surely, as } n\to\infty,
\end{align*}
with $\der_{\empbary}\phi_n$ playing the role of \Frechet derivative of $\phi_n$ at $\empbary$.
Writing out the above equation more explicitly, while noticing that the fixed point characterisation yields $\phi_n(\empbary) = 0$ for every $n$, we then hope to obtain:
\begin{align*}
    \frac{1}{n} \sum_{j=1}^n (\bary^{\sfrac{1}{2}}\Sigma \bary^{\sfrac{1}{2}})^{\sfrac{1}{2}} - \bary =  \der_{\empbary}\phi_n (\bary- \empbary) + o(\bary- \empbary),\qquad\mbox{almost  surely, as } n\to\infty.
\end{align*}
In particular, in light of \eqref{eq : pop bary, fixed point} and the classical CLT, to conclude that $n^{-\sfrac{1}{2}}(\bary- \empbary)$ converges to a Gaussian element on the space of trace class operators, it will suffice to show bounded invertibility of the \Frechet derivative operator $\der_{\empbary}\phi_n $. 
 In the following we transform this heuristic into a rigorous proof.

\medskip




Let us begin with some operator-theoretic preliminaries.
For $F\in\CovSpacePos$ let us define the operator-of-operators $\G{F}(\cdot)$ is defined by:
$$
H \mapsto \G{F}(H):= \int_{0}^\infty\left( e^{-t F^{\sfrac{1}{2}}}He^{-t F^{\sfrac{1}{2}}}H\right)\diff t
$$
In finite dimensions, for any $F\succ 0$ the operator $\G{F}$ is the \Frechet derivative of the square root functional $F\mapsto F^{\sfrac{1}{2}}$. However, such operator is generally unbounded in infinite dimensions, so that it is improper to discuss its \Frechet differentiability. With this in mind, we next establish a weaker differentiability result for a regularised square root operation, which will be sufficient and instrumental for the Taylor-expansion argument that will yield our central limit theorem. First, we give some simple but important properties of $\G{F}$.

\begin{lemma}\label{lemma : sqrt der}
The following hold:
\begin{enumerate}
\setlength{\itemindent}{+2em}
    \item[(i)] For $F\in\CovSpace_+$,  $\G{F}(\cdot)$  satisfies the Sylvester equation:
        \begin{equation}\label{eq:sylvester}
        F^{\sfrac{1}{2}}\G{F}(H) + \G{F}(H)F^{\sfrac{1}{2}} = H, \qquad H\in\Bscr_2
        \end{equation}

\end{enumerate}
Let $\{(e_i\}_{i\geq 1}$ be a CONS of eigenvectors for $F$ with eigenvalues sequence $\{\lambda_i\}_{i\geq 1}$.
\begin{enumerate}
\setlength{\itemindent}{+2em}  \setcounter{enumi}{1}

    \item[(ii)] For $F\in\CovSpace_+$, $\G{F}(\cdot)$ is densely defined by:
        \begin{equation}\label{eq : sqrt id}
            \langle H e_i, e_j \rangle 
           = (\lambda_i^{\sfrac{1}{2}} + \lambda_j^{\sfrac{1}{2}} )\cdot\langle \G{F} (H) e_i, e_j \rangle.
        \end{equation}

    \item[(iii)] $\G{F}(\cdot)$ commutes with the operators $H\mapsto F^\alpha H$ and $H\mapsto HF^\alpha $, for any $\alpha>0$.

    \item[(iv)] For any $G\in\CovSpace\::\: G\preceq F$ we have that:
        \begin{equation}\label{eq: op_bound_sqrt}
        \| \G{F}(G^{\sfrac{1}{2}}H)\|_2\leq \|H\|_2.
        \end{equation}

    \item[(v)] If $0\preceq A \preceq B$ then:
        \begin{equation}\label{eqn:G increasing}
        \forall \: H\in\Bscr_2  \: : \quad \langle  \G{A^2}(AHA) , H \rangle_2 \leq \langle  \G{B^2}(BHB) , H \rangle_2.
        \end{equation}

    \item[(vi)] For fixed $\Sigma\in\CovSpace$, the map: $g\,:\,F \mapsto (F\Sigma F)^{\sfrac{1}{2}}$
is \Frechet differentiable on the space of self-adjoint Hilbert-Schmidt operators, with \Frechet derivative given by:
\begin{equation}\label{eq:sqrtder}
H\mapsto \G{F\Sigma F}(F\Sigma H + H\Sigma F)
\end{equation}

\end{enumerate}

\end{lemma}

\begin{proof}[Proof of Lemma~\ref{lemma : sqrt der}]

    Let us begin by proving (i). Consider the map $$t\mapsto  Z_t:= e^{-t F^{\sfrac{1}{2}}}He^{-t F^{\sfrac{1}{2}}}H,$$ and observe that $\G{F}(H) = \int_{0}^{\infty}Z_t\diff t$. Differentiating $Z_t$ wrt $t$ gives:
    $$
    \partial_t Z_t =  -F^{\sfrac{1}{2}}e^{-t F^{\sfrac{1}{2}}}He^{-t F^{\sfrac{1}{2}}} - e^{-t F^{\sfrac{1}{2}}}He^{-t F^{\sfrac{1}{2}}} F^{\sfrac{1}{2}}.
    $$
    Now, note that $\int_0^{\infty} \partial_t Z_t = \lim_{t\to\infty}Z_t - Z_0 = - H$ since $F\succ 0$. Therefore:
    $$
    H = F^{\sfrac{1}{2}}\G{F}(H) + \G{F}(H)F^{\sfrac{1}{2}}
    $$
    establishig the Sylvester equation \eqref{eq:sylvester}.

    \medskip

    (ii) is a simple consequence of (i). Indeed, let $\{e_i\}_{i\geq 1}$ denote a CONS of eigenfunctions for $F$, with corresponding positive eigenvalues $\{\lambda_i\}_{i\geq 1}$.
    For any $i,j\geq 1$:
    \begin{align*}
    \langle H e_i, e_j \rangle 
    &= \langle F^{\sfrac{1}{2}}\G{F} (H) e_i, e_j \rangle  + \langle \G{F} (H) F^{\sfrac{1}{2}} e_i, e_j \rangle\\
    &= \langle \G{F} (H) e_i, F^{\sfrac{1}{2}}e_j \rangle  + \langle \G{F} (H) F^{\sfrac{1}{2}} e_i, e_j \rangle 
    = (\lambda_i^{\sfrac{1}{2}} + \lambda_j^{\sfrac{1}{2}} )\cdot\langle \G{F} (H) e_i, e_j \rangle  
    \end{align*}

    \medskip

    (iii) is a simple consequence of (ii) Indeed, by \eqref{eq : sqrt id}, for any $\alpha>0$:
        $$
        (\lambda_i^{\sfrac{1}{2}} + \lambda_j^{\sfrac{1}{2}} )\cdot\langle \G{F} (F^{\alpha}H) e_i, e_j \rangle 
        = \lambda_j^{\alpha}\langle H e_i, e_j \rangle
       = (\lambda_i^{\sfrac{1}{2}} + \lambda_j^{\sfrac{1}{2}} )\cdot\langle F^{\alpha}\G{F} (H) e_i, e_j \rangle.
        $$

    \medskip

    (iv) follows by (ii), using that $G\prec F$ implies by Lemma~\ref{lemma : A<B} that there exists a $J$ with $\|J\|_{\infty}\leq 1$ such that $G^{\sfrac{1}{2}} = F^{\sfrac{1}{2}}J$:
   \begin{align*}
    \|\G{F}(G^{\sfrac{1}{2}} H)\|_{2}^2
     &= \sum_{i,j}\langle\G{F}(F^{\sfrac{1}{2}}J H) e_i, e_j\rangle^2 \\
    &=  \sum_{i,j}\frac{\lambda_j}{(\lambda_i^{1/2} + \lambda_j^{1/2})^2}\langle J H e_i,e_j\rangle^2 \\
    &\leq  \sum_{i,j}\langle JH e_i, e_j\rangle^2 \leq \| H\|_{2}^2,
    \end{align*}

    \medskip

   (v) follows from the fact that the inverse map:
$(H\mapsto  \G{A^2}(AHA))^{-1} = (X \mapsto  A^{-1}X+XA^{-1}$
is decreasing in $A$. Indeed, whenever $A\preceq B$:
\begin{align*}
\langle A^{-1}X+XA^{-1} , \: X \rangle 
= \:
2\trace(XA^{-1}X) 
=& \:\sum_{j\geq 1}\langle A^{-1} (Xe_j), (Xe_j)\rangle
\\\geq& \:
 \sum_{j\geq 1}\langle B^{-1} (Xe_j), (Xe_j)\rangle
\\=&\:
 2\trace(XB^{-1}X)
= \:
\langle B^{-1}X+XB^{-1} , \: X \rangle, 
\end{align*}

    \medskip

    We conclude by showing the differentiablity statement in (vi).
    For any $F, H\in\Bscr_2^*$, as in \cite{del2018taylor}, we note that the following identity holds:
    \begin{equation}\label{eqn:CFG}
    \begin{split}
    \CFG \: :=& \: \left(g(F)(g(F+H) - g(F))+(g(F+H) + g(F)) g(F)\right)
    \\ = &\:  - (g(F+H)-g(F))^2 + (g(F+H)^2 - g(F)^2)
    \end{split}
    \end{equation}
    However, by Sylvester's equation \eqref{eq:sylvester}, the quantity  $\CFG$ can also be written as:
    \begin{align*}
    \CFG \:&=\: (F\Sigma F)^{\sfrac{1}{2}}\G{F\Sigma F}(\CFG) + \G{F\Sigma F}(\CFG)(F\Sigma F)^{\sfrac{1}{2}}
    \\ &=\: g(F)\G{F\Sigma F}(\CFG) + \G{F\Sigma F}(\CFG)g(F)
    \end{align*}
    by which necessarily $g(F+H)- g(F) = \G{F\Sigma F}(\CFG)$, 
    and hence we obtain the following decomposition
    \begin{align*}
    g(F+H)- g(F) &=\G{F\Sigma F}\left[ g(F+H)^2 - g(F)^2\right] -\G{F\Sigma F}\left[(g(F+H)-g(F))^2 \right]
    \\ &=\G{F\Sigma F}\left[ F\Sigma H + H\Sigma F\right] + o(H)
    \end{align*}
    proving that that $g$ is differentiable, with derivative given by \eqref{eq:sqrtder}.


    
\end{proof}

The main ingredient in the proof of the Central limit theorem is the following approximation for a fixed point functional on the space of Hilbert-Schmidt operators.

\begin{lemma}\label{lemma : sample fix point functional and properties -- HS -- 2}
    Let $\Sigma_1,\ldots,\Sigma_n$ be i.i.d.\ copies of a random element  $\Sigma$ in $\CovSpace(\Hbb)$ satisfying~\ref{B2}, for $\Hbb$ a separable (possibly infinite dimensional) Hilbert space. Let $\bary\in\CovSpace$ be the (unique) \Frechet mean of $\Sigma$,  and $\{\empbary\}_{n\geq 1}$ a sequence of empirical \Frechet means. Then:
\begin{equation}\label{eq:the_1_4_CLT}
       \frac{1}{n}\sum_{j=1}^n (\bary^{\sfrac{1}{2}} \Sigma_j\bary^{\sfrac{1}{2}})^{\sfrac{1}{2}} - \bary  = 
\frac{1}{n}\sum_{j=1}^n \G{{\bary^{\sfrac{1}{2}}\Sigma_j\bary^{\sfrac{1}{2}}}} (\bary^{\sfrac{1}{2}}T_{\bary}^{\Sigma_j}(\bary^{\sfrac{1}{2}} \H + \H\bary^{\sfrac{1}{2}})T_{\bary}^{\Sigma_j}\bary^{\sfrac{1}{2}}) + R_n,
\end{equation}
where $\H:= \bary^{\sfrac{1}{2}} - \empbary^{\sfrac{1}{2}} $, and $\bary^{\sfrac{1}{2}}R_n, R_n\bary^{\sfrac{1}{2}} = o(\H) $ as $n \to \infty$.
\end{lemma}

\begin{proof}[Proof of Lemma~\ref{lemma : sample fix point functional and properties -- HS -- 2}]
    For $t\in[0,1]$ define $F_t := \bary^{\sfrac{1}{2}} + t\H$, with $\H:=(\empbary^{\sfrac{1}{2}} - \bary^{\sfrac{1}{2}})$. Consider the map:
$$
\phi_n\,:\, [0,1] \mapsto \Bscr_2, \qquad t \mapsto 
=  \frac{1}{n}\sum_{i=1}^n 
(F_t\Sigma_j F_t)^{\sfrac{1}{2}} - F_t^2.
$$
By differentiablity statement (vi) in Lemma~\ref{lemma : sqrt der}:
\begin{equation*}\label{eqn:lemmapreclt:sum}
\begin{split}
  \phi_n'(t) \:&=\:
  \frac{1}{n}\sum_{i=1}^n 
     \left(\G{ Q_{j,t}^2 } (\H \Sigma_j F_t +  F_t\Sigma_j\H) \right) - F_t\H - \H F_t
     \\&=\: 
  \frac{1}{n}\sum_{i=1}^n 
     \left(\G{ Q_{j,t}^2 } (\H F_t^{-1} Q_{j,t}^2 +  Q_{j,t}^2F_t^{-1}\H) \right) - F_t\H - \H F_t
     \\&=\: 
    \frac{1}{n}\sum_{j=1}^n\left(\G{ Q_{j,t}^2 } ( (\H F_t^{-1} Q_{j,t})Q_{j,t}  +  Q_{j,t}(\H F_t^{-1} Q_{j,t})^* ) \right) - F_t\H - \H F_t
\end{split}
\end{equation*}
where $Q_{j,t}:= (F_t\Sigma_jF_t)^{\sfrac{1}{2}} = F_tT_{j,t}F_t$ and $T_{j,t}:=T_{F_t^2}^{\Sigma_j}$, and $t\in[0,1]$. 
It is easy to see that, by \eqref{eq:sylvester}, for all $X\in \Bscr_2$ and $j$:
$$
\G{Q_{j,t}^2}(Q_{j,t}\X +\X^* Q_{j,t})= X+X^* - \G{Q_{j,t}^2}(\X Q_{j,t} + Q_{j,t}\X^*)
$$
Setting $\X = \H F_t^{-1} Q_{j,t}$, and noticing that $\X^*= F_tT_{j,t}\H$,  shows in fact:
\begin{align*}
 \phi_n'(t) \:=& \: -\frac{1}{n}\sum_{j=1}^n\G{ Q_{j,t}^2 } (Q_{j,t} ( \H F_t^{-1}  + F_t^{-1} \H ) Q_{j,t}) 
 \\ &\quad+ \frac{1}{n}\sum_{j=1}^n\left( F_t(T_{j,t}-\Id)\H + \H(T_{j,t}-\Id)F_t\right)
\end{align*} 
which is easily shown to be bounded uniformly in $t\in[0,1]$.
We may thus extend the Taylor polynomial with integral reminder, and obtain:
\begin{align*}
\frac{1}{n}\sum_{j=1}^n (\bary^{\sfrac{1}{2}} \Sigma_j\bary^{\sfrac{1}{2}})^{\sfrac{1}{2}} - \bary  =& 
-\int_{0}^{1}\frac{1}{n}\sum_{j=1}^n\G{ Q_{j,t}^2 } (Q_{j,t} ( \H F_t^{-1}  + F_t^{-1} \H ) Q_{j,t}) \diff t
 \\ &\quad+ \int_{0}^{1}\frac{1}{n}\sum_{j=1}^n\left( F_t(T_{j,t}-\Id)\H + \H(T_{j,t}-\Id)F_t\right)\diff t
\end{align*}
where we have used the fixed-point characterisation \eqref{eqn:pop-fixed-pt}.
Note that $F_t\preceq ((1-t) + t \|\bary\|^{\sfrac{1}{2}})\bary^{\sfrac{1}{2}}$. Hence, by Lemma~\ref{lemma : A<B}, using continuity, and the fact that \eqref{eq : pop bary, fixed point on map} implies $\frac{1}{n}\sum_{j=1}^nT_{\bary}^{\Sigma_j}\H = \H + o(\H)$ as $n\to\infty$, we readily obtain that there exists $R_n$ such that
\begin{equation}
    \label{eq:the_1_4_CLT_final}
\frac{1}{n}\sum_{j=1}^n (\bary^{\sfrac{1}{2}} \Sigma_j\bary^{\sfrac{1}{2}})^{\sfrac{1}{2}} - \bary  = 
-\frac{1}{n}\sum_{j=1}^n\G{ Q_{j}^2 } (Q_{j} ( \H \bary^{-\sfrac{1}{2}}  + \bary^{-\sfrac{1}{2}} \H ) Q_{j}) + R_n
\end{equation}
where $Q_j = Q_{j,0}=(\bary^{\sfrac{1}{2}}\Sigma_j\bary^{\sfrac{1}{2}})^{\sfrac{1}{2}}$ and $R_n = R_n(\H)$ satisfies:
\begin{equation}\label{eq : estimating the remainder}
\bary^{-\sfrac{1}{2}}R_n = o(\H) \quad \text{and} \quad  R_n\bary^{-\sfrac{1}{2}}= o(\H),\qquad\text{ as } n \to \infty.
\end{equation}
\end{proof}

We can finally state and prove the central limit theorem.

\begin{theorem}\label{thm : clt}
 Let $\Sigma_1,\ldots,\Sigma_n$ be i.i.d.\ copies of a random element  $\Sigma$ in $\CovSpace(\Hbb)$ satisfying~\ref{B2},  for  $\Hbb$ a separable (possibly infinite dimensional) Hilbert space. Let $\bary\in\CovSpace$ be the (unique) \Frechet mean  of $\Sigma$, and $\{\empbary\}_{n\geq 1}$ the sequence of empirical \Frechet means. Then:
 \begin{equation}
    \label{eq:CLTthm}    
\Pi( \bary, \empbary) = O(n^{-\sfrac{1}{2}}), \quad \quad \mbox{almost surely},
\end{equation}
and the scaled difference $\sqrt{n}(\bary^{\sfrac{1}{2}}-\empbary^{\sfrac{1}{2}})$ converges weakly to a {centred} Gaussian random element in the space $\Bscr_2$ of Hilbert-Schmidt operators.
\end{theorem}
\begin{proof}[Proof of Theorem~\ref{thm : clt}]
Set $\H= \empbary^{\sfrac{1}{2}} - \bary^{\sfrac{1}{2}}$. 
For tidiness, let us also write $T_j:= T_{\bary}^{\Sigma_j}$ and $Q_j:=(\bary^{\sfrac{1}{2}}\Sigma_j\bary^{\sfrac{1}{2}})^{\sfrac{1}{2}} = \bary^{\sfrac{1}{2}}T_j\bary^{\sfrac{1}{2}}$ for $j=1,\dots,n$ and $n\geq 1$. 
Consider the expansion in \eqref{eq:the_1_4_CLT_final}.
Assumption~\ref{B2} implies via Lemma~\ref{lemma:boundbothmaps} that the optimal maps $T_j$ are boundedly invertible, and in particular that:
$$
\forall\:h\in\Hbb\::\quad \langle Q_jh,h\rangle = \langle \bary^{\sfrac{1}{2}} T_j\bary^{\sfrac{1}{2}} h,h\rangle = \|T_j^{\sfrac{1}{2}}\bary^{\sfrac{1}{2}} h\| \: \gtrsim \|\bary^{\sfrac{1}{2}} h\|\: = \langle \bary h,h\rangle.
$$
i.e.\ that $ \bary \precsim Q_j$ almost surely.
Consequently, by \eqref{eqn:G increasing}, we obtain that:
$$
  H \mapsto \frac{1}{n}\sum_{j=1}^n  \G{Q_j^2}(Q_j(\bary^{-\sfrac{1}{2}}( \bary^{\sfrac{1}{2}}H + H \bary^{\sfrac{1}{2}})\bary^{-\sfrac{1}{2}})Q_j)  
\: \succeq \:
 H \mapsto \G{\bary^2}(\bary^{\sfrac{1}{2}}( \bary^{\sfrac{1}{2}}H + H \bary^{\sfrac{1}{2}})\bary^{\sfrac{1}{2}}).
$$
as operators on $\Bscr_2$.
Note that the inverse of the right hand term is given by:
$$
X\mapsto \G{\bary}(\bary^{\sfrac{1}{2}}X\bary^{-\sfrac{1}{2}} + \bary^{-\sfrac{1}{2}}X\bary^{\sfrac{1}{2}}).
$$
Indeed, composing the first with the second:
\begin{align*}
     \G{\bary^2}&\left(\bary^{\sfrac{1}{2}}\left( \bary^{\sfrac{1}{2}}\G{\bary}(\bary^{\sfrac{1}{2}}X\bary^{-\sfrac{1}{2}} + \bary^{-\sfrac{1}{2}}X\bary^{\sfrac{1}{2}}) + \G{\bary}(\bary^{\sfrac{1}{2}}X\bary^{-\sfrac{1}{2}} + \bary^{-\sfrac{1}{2}}X\bary^{\sfrac{1}{2}}) \bary^{\sfrac{1}{2}}\right)\bary^{\sfrac{1}{2}}\right) \\
      =&\: \G{\bary^2}\left( \bary^{\sfrac{1}{2}}\G{\bary}(\bary X + X\bary) + \G{\bary}(\bary X + X\bary) \bary^{\sfrac{1}{2}})\right) \\
      =& \:\G{\bary^2}\left( \bary( \bary^{\sfrac{1}{2}}\G{\bary}( X) + \G{\bary}( X)\bary^{\sfrac{1}{2}}) + ( \bary^{\sfrac{1}{2}}\G{\bary}( X) + \G{\bary}( X)\bary^{\sfrac{1}{2}}) \bary)\right) \\
       =&\: \G{\bary^2}\left( \bary X + X\bary\right) 
      = \:\bary \G{\bary^2}(X) + \G{\bary^2}(X) \bary 
       = \: X 
\end{align*}
where we have used Sylverster's identity twice, and repeatedly that $\G{\bary}$ commutes with $\bary^{\sfrac{1}{2}}$ as stated in Lemma~\ref{lemma : sqrt der}. We  now rewrite \eqref{eq:the_1_4_CLT_final} as:
$$
\frac{1}{n}\sum_{j=1}^n (\bary^{\sfrac{1}{2}} \Sigma_j\bary^{\sfrac{1}{2}})^{\sfrac{1}{2}} - \bary  = 
\Acov_n(\H) + R_n.
$$
The argument just presented shows that $\Acov_n$ admits a left-inverse $\Acov_n^{-1}$ which is \textit{uniformly} bounded (in $n$) on all terms appearing in \eqref{eq:the_1_4_CLT_final}. In particular:
\begin{align*}
\left \| \Acov_n^{-1}\left(\frac{1}{n}\sum_{j=1}^n (\bary^{\sfrac{1}{2}} \Sigma_j\bary^{\sfrac{1}{2}})^{\sfrac{1}{2}} - \bary\right) \right\|_{2} \leq & \:
\left \|
\G{\bary^2}\left(
\bary \left(\frac{1}{n}\sum_{j=1}^n T_i - \Id\right)\right)\right\|_{2}
\\\leq&\: \text{const.}\cdot
\left \|
\bary^{\sfrac{1}{2}} \frac{1}{n}\left(\sum_{j=1}^n T_i - \Id\right)\right\|_{2}
\end{align*}
Also notice that, by the bound on the Taylor remainder \eqref{eq : estimating the remainder}, we have that $$\left\|\Acov_n^{-1}(R_n)\right\|_{2} \leq \text{const.}\cdot\left\| \G{\bary}(\bary^{\sfrac{1}{2}}R_n\bary^{-\sfrac{1}{2}} + \bary^{-\sfrac{1}{2}}R_n\bary^{\sfrac{1}{2}})\right\|_{2} = o(\|\H\|_{2}).
$$
Therefore, applying such inverse to \eqref{eq:the_1_4_CLT_final} and scaling by $n^{\sfrac{1}{2}}$ gives:
\begin{align*}
\| \H\|_{2}  + o(\H)
\: \leq\:\text{const.}\cdot\left\| n^{-1}\sum_{j=1}^n   \left(T_{j} - \Id\right)\bary^{\sfrac{1}{2}}\right\|_{2}
\end{align*}
which establishes the rate of convergence, employing the CLT on the Hilbert space $\Bscr_2$ for the sequence of i.i.d.\ mean-zero objects $\{T_{j}\bary^{\sfrac{1}{2}} - \bary^{\sfrac{1}{2}}\}_{j\geq 1}$. 

Finally, it is easy to see that $\Acov_n(\cdot) \to \Acov(\cdot)$ strongly, almost surely, as $n\to\infty$,  where 
$$\Acov(H):=\Ebb\left[ \G{{\bary^{\sfrac{1}{2}}\Sigma\bary^{\sfrac{1}{2}}}} (\bary^{\sfrac{1}{2}}T_{\bary}^{\Sigma}(\bary^{\sfrac{1}{2}} H + H\bary^{\sfrac{1}{2}})T_{\bary}^{\Sigma}\bary^{\sfrac{1}{2}}) \right].$$
By linearity, boudedness and the convergence of the sequence of maps $\Acov_n$, the Gaussian limit in distribution is preserved, which completes the proof.
\end{proof}

\section*{Acknowledgments}
This research was supported by a Swiss NSF research grant.

\bibliography{bib.bib}       

\begin{thebibliography}{}

\bibitem[\protect\citeauthoryear{Afsari}{Afsari}{2011}]{afsari2011riemannian}
Afsari, B. (2011).
\newblock Riemannian $l^{p}$ center of mass: existence, uniqueness, and
  convexity.
\newblock {\em Proceedings of the American Mathematical Society\/}~{\em
  139\/}(2), 655--673.

\bibitem[\protect\citeauthoryear{Agueh and Carlier}{Agueh and
  Carlier}{2011}]{agueh2011barycenters}
Agueh, M. and G.~Carlier (2011).
\newblock Barycenters in the {W}asserstein space.
\newblock {\em SIAM Journal on Mathematical Analysis\/}~{\em 43\/}(2),
  904--924.

\bibitem[\protect\citeauthoryear{{\'A}lvarez-Esteban, Del~Barrio,
  Cuesta-Albertos, and Matr{\'a}n}{{\'A}lvarez-Esteban
  et~al.}{2016}]{alvarez2016fixed}
{\'A}lvarez-Esteban, P.~C., E.~Del~Barrio, J.~Cuesta-Albertos, and
  C.~Matr{\'a}n (2016).
\newblock A fixed-point approach to barycenters in wasserstein space.
\newblock {\em Journal of Mathematical Analysis and Applications\/}~{\em
  441\/}(2), 744--762.

\bibitem[\protect\citeauthoryear{{\'A}lvarez-Esteban, del Barrio,
  Cuesta-Albertos, and Matr{\'a}n}{{\'A}lvarez-Esteban
  et~al.}{2018}]{alvarez2018wide}
{\'A}lvarez-Esteban, P.~C., E.~del Barrio, J.~A. Cuesta-Albertos, and
  C.~Matr{\'a}n (2018).
\newblock {Wide consensus aggregation in the {W}asserstein space: application
  to location-scatter families}.
\newblock {\em Bernoulli\/}~{\em 24\/}(4A), 3147 -- 3179.

\bibitem[\protect\citeauthoryear{Ambrosio, Gigli, and Savar{\'e}}{Ambrosio
  et~al.}{2005}]{ambrosio2005gradient}
Ambrosio, L., N.~Gigli, and G.~Savar{\'e} (2005).
\newblock {\em Gradient flows: in metric spaces and in the space of probability
  measures}.
\newblock Springer Science \& Business Media.

\bibitem[\protect\citeauthoryear{Baker}{Baker}{1970}]{baker1970covariance}
Baker, C.~R. (1970).
\newblock On covariance operators.
\newblock Technical report, North Carolina State University. Dept. of
  Statistics.

\bibitem[\protect\citeauthoryear{Benko, H{\"a}rdle, and Kneip}{Benko
  et~al.}{2009}]{benko2009common}
Benko, M., W.~H{\"a}rdle, and A.~Kneip (2009).
\newblock Common functional principal components.

\bibitem[\protect\citeauthoryear{Bhatia, Jain, and Lim}{Bhatia
  et~al.}{2019}]{bhatia2019bures}
Bhatia, R., T.~Jain, and Y.~Lim (2019).
\newblock On the {B}ures--{W}asserstein distance between positive definite
  matrices.
\newblock {\em Expositiones Mathematicae\/}~{\em 37\/}(2), 165--191.

\bibitem[\protect\citeauthoryear{Bhattacharya and Patrangenaru}{Bhattacharya
  and Patrangenaru}{2003}]{bhattacharya2003large}
Bhattacharya, R. and V.~Patrangenaru (2003).
\newblock Large sample theory of intrinsic and extrinsic sample means on
  manifolds.
\newblock {\em The Annals of Statistics\/}~{\em 31\/}(1), 1--29.

\bibitem[\protect\citeauthoryear{Bogachev}{Bogachev}{1998}]{bogachev1998gaussian}
Bogachev, V.~I. (1998).
\newblock {\em Gaussian measures}.
\newblock Number~62. American Mathematical Soc.

\bibitem[\protect\citeauthoryear{Br{\'e}zis}{Br{\'e}zis}{2011}]{brezis2011functional}
Br{\'e}zis, H. (2011).
\newblock {\em Functional analysis, Sobolev spaces and partial differential
  equations}, Volume~2.
\newblock Springer.

\bibitem[\protect\citeauthoryear{Chen, Lin, and M{\"u}ller}{Chen
  et~al.}{2021}]{chen2021wasserstein}
Chen, Y., Z.~Lin, and H.-G. M{\"u}ller (2021).
\newblock Wasserstein regression.
\newblock {\em Journal of the American Statistical Association\/}, 1--14.

\bibitem[\protect\citeauthoryear{Cuesta-Albertos, Matr{\'a}n-Bea, and
  Tuero-Diaz}{Cuesta-Albertos et~al.}{1996}]{cuesta1996lower}
Cuesta-Albertos, J., C.~Matr{\'a}n-Bea, and A.~Tuero-Diaz (1996).
\newblock On lower bounds for the l2-{W}asserstein metric in a {H}ilbert space.
\newblock {\em Journal of Theoretical Probability\/}~{\em 9\/}(2), 263--283.

\bibitem[\protect\citeauthoryear{Dai, Zhang, and Srivastava}{Dai
  et~al.}{2019}]{dai2019analyzing}
Dai, M., Z.~Zhang, and A.~Srivastava (2019).
\newblock Analyzing dynamical brain functional connectivity as trajectories on
  space of covariance matrices.
\newblock {\em IEEE transactions on medical imaging\/}~{\em 39\/}(3), 611--620.

\bibitem[\protect\citeauthoryear{Del~Moral and Niclas}{Del~Moral and
  Niclas}{2018}]{del2018taylor}
Del~Moral, P. and A.~Niclas (2018).
\newblock A taylor expansion of the square root matrix function.
\newblock {\em Journal of Mathematical Analysis and Applications\/}~{\em
  465\/}(1), 259--266.

\bibitem[\protect\citeauthoryear{Dryden, Koloydenko, and Zhou}{Dryden
  et~al.}{2009}]{dryden2009non}
Dryden, I.~L., A.~Koloydenko, and D.~Zhou (2009).
\newblock Non-{E}uclidean statistics for covariance matrices, with applications
  to diffusion tensor imaging.
\newblock {\em The Annals of Applied Statistics\/}~{\em 3\/}(3), 1102--1123.

\bibitem[\protect\citeauthoryear{Evans and Jaffe}{Evans and
  Jaffe}{2023}]{evans2020strong}
Evans, S.~N. and A.~Q. Jaffe (2023+).
\newblock Strong laws of large numbers for {F}r\'{e}chet means.
\newblock {\em Bernoulli (to appear, available at arXiv:2012.12859)\/}.

\bibitem[\protect\citeauthoryear{Ferraty and Vieu}{Ferraty and
  Vieu}{2006}]{ferraty2006nonparametric}
Ferraty, F. and P.~Vieu (2006).
\newblock {\em Nonparametric functional data analysis}.
\newblock Springer.

\bibitem[\protect\citeauthoryear{Fr{\'e}chet}{Fr{\'e}chet}{1948}]{frechet1948elements}
Fr{\'e}chet, M. (1948).
\newblock Les {\'e}l{\'e}ments al{\'e}atoires de nature quelconque dans un
  espace distanci{\'e}.
\newblock In {\em Annales de l'institut Henri Poincar{\'e}}, Volume~10, pp.\
  215--310.

\bibitem[\protect\citeauthoryear{Gill}{Gill}{2004}]{gill2004teleportation}
Gill, R.~D. (2004).
\newblock Teleportation into quantum statistics.
\newblock {\em arXiv preprint math/0405572\/}.

\bibitem[\protect\citeauthoryear{Hsing and Eubank}{Hsing and
  Eubank}{2015}]{hsing2015theoretical}
Hsing, T. and R.~Eubank (2015).
\newblock {\em Theoretical foundations of functional data analysis, with an
  introduction to linear operators}, Volume 997.
\newblock John Wiley \& Sons.

\bibitem[\protect\citeauthoryear{Kantorovich}{Kantorovich}{1942}]{kantorovich1942translocation}
Kantorovich, L.~V. (1942).
\newblock On the translocation of masses.
\newblock In {\em Dokl. Akad. Nauk. USSR (NS)}, Volume~37, pp.\  199--201.

\bibitem[\protect\citeauthoryear{Kendall and Le}{Kendall and
  Le}{2011}]{kendall2011limit}
Kendall, W.~S. and H.~Le (2011).
\newblock Limit theorems for empirical {F}r{\'e}chet means of independent and
  non-identically distributed manifold-valued random variables.

\bibitem[\protect\citeauthoryear{Koner and Staicu}{Koner and
  Staicu}{2023}]{koner2023second}
Koner, S. and A.-M. Staicu (2023).
\newblock Second-generation functional data.
\newblock {\em Annual Review of Statistics and Its Application\/}~{\em 10},
  547--572.

\bibitem[\protect\citeauthoryear{Kroshnin, Spokoiny, and Suvorikova}{Kroshnin
  et~al.}{2021}]{kroshnin2021statistical}
Kroshnin, A., V.~Spokoiny, and A.~Suvorikova (2021).
\newblock Statistical inference for {B}ures--{W}asserstein barycenters.
\newblock {\em The Annals of Applied Probability\/}~{\em 31\/}(3), 1264--1298.

\bibitem[\protect\citeauthoryear{Le~Gouic and Loubes}{Le~Gouic and
  Loubes}{2017}]{le2017existence}
Le~Gouic, T. and J.-M. Loubes (2017).
\newblock Existence and consistency of {W}asserstein barycenters.
\newblock {\em Probability Theory and Related Fields\/}~{\em 168\/}(3),
  901--917.

\bibitem[\protect\citeauthoryear{Le~Gouic, Paris, Rigollet, and
  Stromme}{Le~Gouic et~al.}{2022}]{le2022fast}
Le~Gouic, T., Q.~Paris, P.~Rigollet, and A.~J. Stromme (2022).
\newblock Fast convergence of empirical barycenters in {A}lexandrov spaces and
  the {W}asserstein space.
\newblock {\em Journal of the European Mathematical Society\/}.

\bibitem[\protect\citeauthoryear{Ledoux and Talagrand}{Ledoux and
  Talagrand}{2013}]{ledoux2013probability}
Ledoux, M. and M.~Talagrand (2013).
\newblock {\em Probability in Banach Spaces: isoperimetry and processes}.
\newblock Springer Science \& Business Media.

\bibitem[\protect\citeauthoryear{Lindquist}{Lindquist}{2008}]{lindquist2008statistical}
Lindquist, M.~A. (2008).
\newblock The statistical analysis of fmri data.
\newblock {\em Statistical science\/}~{\em 23\/}(4), 439--464.

\bibitem[\protect\citeauthoryear{Masarotto, Panaretos, and Zemel}{Masarotto
  et~al.}{2019}]{masarotto2019procrustes}
Masarotto, V., V.~M. Panaretos, and Y.~Zemel (2019).
\newblock Procrustes metrics on covariance operators and optimal transportation
  of {G}aussian processes.
\newblock {\em Sankhya A\/}~{\em 81\/}(1), 172--213.

\bibitem[\protect\citeauthoryear{Masarotto, Panaretos, and Zemel}{Masarotto
  et~al.}{2022}]{masarotto2022transportation}
Masarotto, V., V.~M. Panaretos, and Y.~Zemel (2022).
\newblock Transportation-based functional {ANOVA} and {PCA} for covariance
  operators.
\newblock {\em arXiv preprint arXiv:2212.04797\/}.

\bibitem[\protect\citeauthoryear{Minh}{Minh}{2022}]{minh2022entropic}
Minh, H.~Q. (2022).
\newblock Entropic regularization of {W}asserstein distance between
  infinite-dimensional {G}aussian measures and {G}aussian processes.
\newblock {\em Journal of Theoretical Probability\/}, 1--96.

\bibitem[\protect\citeauthoryear{Olkin and Pukelsheim}{Olkin and
  Pukelsheim}{1982}]{olkin1982distance}
Olkin, I. and F.~Pukelsheim (1982).
\newblock The distance between two random vectors with given dispersion
  matrices.
\newblock {\em Linear Algebra and its Applications\/}~{\em 48}, 257--263.

\bibitem[\protect\citeauthoryear{Panaretos, Kraus, and Maddocks}{Panaretos
  et~al.}{2010}]{panaretos2010second}
Panaretos, V.~M., D.~Kraus, and J.~H. Maddocks (2010).
\newblock Second-order comparison of gaussian random functions and the geometry
  of dna minicircles.
\newblock {\em Journal of the American Statistical Association\/}~{\em
  105\/}(490), 670--682.

\bibitem[\protect\citeauthoryear{Panaretos and Tavakoli}{Panaretos and
  Tavakoli}{2013}]{panaretos2013fourier}
Panaretos, V.~M. and S.~Tavakoli (2013).
\newblock Fourier analysis of stationary time series in function space.
\newblock {\em The Annals of Statistics\/}~{\em 41\/}(2), 568--603.

\bibitem[\protect\citeauthoryear{Pedersen}{Pedersen}{1972}]{pedersen1972some}
Pedersen, G.~K. (1972).
\newblock Some operator monotone functions.
\newblock {\em Proceedings of the American Mathematical Society\/}~{\em
  36\/}(1), 309--310.

\bibitem[\protect\citeauthoryear{Petersen and M{\"u}ller}{Petersen and
  M{\"u}ller}{2019}]{petersen2019frechet}
Petersen, A. and H.-G. M{\"u}ller (2019).
\newblock Fr{\'e}chet regression for random objects with euclidean predictors.
\newblock {\em The Annals of Statistics\/}~{\em 47\/}(2), 691--719.

\bibitem[\protect\citeauthoryear{Pigoli, Aston, Dryden, and Secchi}{Pigoli
  et~al.}{2014}]{pigoli2014distances}
Pigoli, D., J.~A. Aston, I.~L. Dryden, and P.~Secchi (2014).
\newblock Distances and inference for covariance operators.
\newblock {\em Biometrika\/}~{\em 101\/}(2), 409--422.

\bibitem[\protect\citeauthoryear{Powers and St{\o}rmer}{Powers and
  St{\o}rmer}{1970}]{Stormer1970}
Powers, R.~T. and E.~St{\o}rmer (1970).
\newblock {Free states of the canonical anticommutation relations}.
\newblock {\em Communications in Mathematical Physics\/}~{\em 16\/}(1), 1 --
  33.

\bibitem[\protect\citeauthoryear{Santoro and Panaretos}{Santoro and
  Panaretos}{2024}]{santoro2023random}
Santoro, L.~V. and V.~M. Panaretos (2024+).
\newblock Statistical inference for bures-wasserstein flows.
\newblock preprint.

\bibitem[\protect\citeauthoryear{Schachermayer and Teichmann}{Schachermayer and
  Teichmann}{2009}]{schachermayer2009characterization}
Schachermayer, W. and J.~Teichmann (2009).
\newblock Characterization of optimal transport plans for the
  {M}onge-{K}antorovich problem.
\newblock {\em Proceedings of the American Mathematical Society\/}~{\em
  137\/}(2), 519--529.

\bibitem[\protect\citeauthoryear{Takatsu}{Takatsu}{2010}]{takatsu2010wasserstein}
Takatsu, A. (2010).
\newblock On {W}asserstein geometry of {G}aussian measures.
\newblock In {\em Probabilistic approach to geometry}, pp.\  463--472.
  Mathematical Society of Japan.

\bibitem[\protect\citeauthoryear{Thanwerdas and Pennec}{Thanwerdas and
  Pennec}{2023}]{PennecGeod}
Thanwerdas, Y. and X.~Pennec (2023).
\newblock Bures–wasserstein minimizing geodesics between covariance matrices
  of different ranks.
\newblock {\em SIAM Journal on Matrix Analysis and Applications\/}~{\em
  44\/}(3), 1447--1476.

\bibitem[\protect\citeauthoryear{Villani}{Villani}{2009}]{villani2009optimal}
Villani, C. (2009).
\newblock {\em Optimal transport: old and new}, Volume 338.
\newblock Springer.

\bibitem[\protect\citeauthoryear{Zemel}{Zemel}{2023}]{zemel2023non}
Zemel, Y. (2023).
\newblock Non-injectivity of bures--wasserstein barycentres in infinite
  dimensions.
\newblock {\em arXiv preprint arXiv:2311.09385\/}.

\bibitem[\protect\citeauthoryear{Zemel and Panaretos}{Zemel and
  Panaretos}{2019}]{zemel2019frechet}
Zemel, Y. and V.~M. Panaretos (2019).
\newblock {F}r{\'e}chet means and {P}rocrustes analysis in {W}asserstein space.
\newblock {\em Bernoulli\/}~{\em 25\/}(2), 932--976.

\bibitem[\protect\citeauthoryear{Ziezold}{Ziezold}{1977}]{ziezold1977expected}
Ziezold, H. (1977).
\newblock On expected figures and a strong law of large numbers for random
  elements in quasi-metric spaces.
\newblock In {\em Transactions of the Seventh Prague Conference on Information
  Theory, Statistical Decision Functions, Random Processes and of the 1974
  European Meeting of Statisticians}, pp.\  591--602. Springer.

\end{thebibliography}

\end{document}